\documentclass[aihp]{imsart}

\RequirePackage{amsthm,amsmath,amsfonts,amssymb,color,bbm}
\RequirePackage[numbers]{natbib}
\RequirePackage[colorlinks,citecolor=blue,urlcolor=blue]{hyperref}
\RequirePackage{graphicx}

\startlocaldefs
\theoremstyle{plain}

\newtheorem{theorem}{Theorem}[section]
\newtheorem{lemma}[theorem]{Lemma}
\newtheorem{corollary}[theorem]{Corollary}

\newtheorem{remark}[theorem]{Remark}

\theoremstyle{remark}

\endlocaldefs

\makeatletter
\g@addto@macro\normalsize{%
  \setlength\abovedisplayskip{8pt plus 3pt minus 3pt}
  \setlength\belowdisplayskip{8pt plus 3pt minus 3pt}
  \setlength\abovedisplayshortskip{6pt plus 3pt minus 2pt}
  \setlength\belowdisplayshortskip{6pt plus 3pt minus 2pt}
}
\makeatother

\def\leq{\leqslant}
\def\geq{\geqslant}
\def\le{\leqslant}
\def\ge{\geqslant}

\numberwithin{equation}{section}

\def\({\bigl(}
\def\){\bigr)}

\usepackage{xcolor}

\def\dfrac#1#2{\lower0.15ex\hbox{\large$\textstyle\frac{#1}{#2}$}}
\def\({\bigl(}
\def\){\bigr)}

\def\obj{H}

\let\eps=\varepsilon

\def\P{\boldsymbol{P}}

\def\HG{\boldsymbol{H}}

\def\X{\boldsymbol{X}}
\def\Y{\boldsymbol{Y}}

\def\Reals{{\mathbb{R}}}

\def\Naturals{{\mathbb{N}}}

\def\Bin{\operatorname{Bin}}

\newcommand{\m}{~\middle|~}
\newcommand{\A}{\mathbf A}

\renewcommand{\P}[1]{\Pr\left(#1\right)}

\newcommand{\E}[1]{\mathbf E\left[#1\right]}
\def\HG{\mathcal{H}}

\renewcommand{\o}[1]{o\left(#1\right)}

\newcommand{\D}{\mathbf D}

\newcommand{\red}[1]{#1}
\newcommand{\purple}[1]{#1}

\begin{document}

\begin{frontmatter}

\title{Extremal independence in discrete random systems }
\runtitle{Extremal independence in discrete random systems}

\begin{aug}
\author[A]{\inits{M.}\fnms{Mikhail}~\snm{Isaev}\ead[label=e1]{mikhail.isaev@monash.edu}},
\author[B]{\inits{I.}\fnms{Igor}~\snm{Rodionov}\ead[label=e2]{igor.rodionov@epfl.ch}
},
\author[A]{\inits{R.}\fnms{Rui-Ray}~\snm{Zhang}\ead[label=e3]{rui.zhang@monash.edu}}
\and
\author[C]{\inits{M.}\fnms{Maksim}~\snm{Zhukovskii}\ead[label=e4]{zhukmax@tauex.tau.ac.il}}
\address[A]{School of Mathematics,
Monash University, Clayton, VIC, Australia \printead[presep={,\ }]{e1,e3}}

\address[B]{Institute of Mathematics,  EPFL, Lausanne, Switzerland 
\printead[presep={,\ }]{e2}}

\address[C]{School of Mathematical Sciences, Tel Aviv University, Tel Aviv, Israel 
\printead[presep={,\ }]{e4}}
\end{aug}

\begin{abstract}
Let $\X(n) \in \Reals^{d}$ be a sequence of random vectors, where  $n\in\mathbb{N}$ and $d = d(n)$.
Under certain weakly dependence conditions, we  prove that the distribution of the maximal component  of $\X$
and the distribution of the   maximum of  their independent copies  are
asymptotically equivalent.
Our result on extremal independence relies on new lower and upper bounds for
the probability that none of a given finite set of events occurs. As applications, we obtain the distribution of   various extremal characteristics of  random discrete structures such as
maximum codegree in binomial random hypergraphs and the maximum number of cliques sharing a given vertex in binomial random graphs.  We also generalise Berman-type conditions for  a sequence of Gaussian random vectors to possess the extremal independence property.
\end{abstract}




\end{frontmatter}

\section{Introduction}

Fisher--Tippett--Gnedenko theorem is central in the extreme value theory;
it was discovered first by Fisher and Tippett \cite{fishertippett}
and later proved in full generality by Gnedenko \cite{gnedenko}.
This theorem states that
if the maximum of the first $n$ terms of a sequence of independent and identically distributed (i.i.d.)
random variables has a non-degenerate limit distribution after a proper normalisation,
then it belongs to either the Gumbel, the Fr\'echet, or the Weibull families of distributions.

The FTG theorem generalises to stationary random sequences
of dependent random variables
under the additional assumptions that its distant terms are independent \cite{watson1954extreme} or weakly dependent \cite{loynes}.
 Leadbetter in \cite{leadbetterself} significantly relaxed the assumptions of \cite{watson1954extreme, loynes}. 
 The analogues of Leadbetter's conditions were also found for non-stationary sequences \cite{husler1, husler} and for random fields \cite{leadbetterrootzen, ling, pereiraferreira}. In fact, the behaviour of maxima for non-stationary sequences is more complicated than that for the stationary case. Even in the simplest case when the variables are independent, the limit distribution might not belong to any of
the Gumbel, the Frechet or the Weibull families, see {\cite[Section 8.3]{falk}}.
That is, it is  impossible to classify  all possible limit distributions for general random systems with dependencies.
Nevertheless, extremal characteristics of such systems  always attracted significant attention of researchers in computer science, statistical physics, financial mathematics and network studies. For example,  in a recent work
 \cite{resnick}, the authors study tail indices of in-degree and out-degree of the nodes of social networks.  However, they lack to justify that their methods such as Hill estimator can be extended to
 non-i.i.d. data.
%
 Similar lack of mathematically rigorous justification occurs in several papers on ranking web pages~\cite{jelenkovic,volkovich}.

In this paper,
we 
focus on the following
extremal independence property that  helps to reduce general random systems
to the independent case, where standard statistical techniques apply.
%
Let $\X(n) =(X_1(n), \ldots, X_{d}(n))^T \in \Reals^{d}$
be a sequence of random vectors,
where $d = d(n)$ be a sequence of positive integers.
We give sufficient conditions for the property that
\begin{equation}
\left| \P{ \max_{i \in [d]} X_i \leq x }
- \prod_{i\in [d]} \P{ X_i \leq x } \right|
\rightarrow 0 \qquad \text{for any fixed $x\in\Reals$.}
\label{a_e_indep_formula}
\end{equation}
All asymptotics in this paper refer to the passage of $n$ to infinity and the notations
$o(\cdot)$, $O(\cdot)$, \red{$\omega(\cdot)$}, $\Omega(\cdot)$, \red{$\Theta(\cdot)$} have the standard meaning.

Allowing arbitrary sequences of vectors   $\X(n)$ in \eqref{a_e_indep_formula}
encapsulates
several similar questions arising in the studies of
sequences of random variables,
triangular arrays, random fields, and so on.
For example, for a sequence $\xi_1, \xi_2, \ldots $ of identically distributed (i.d.) random variables, one can set
\[
X_i(n):= \frac{\xi_i - a_n}{b_n},
\]
 where  $a_n$ and $b_n$ are the normalising constants from the FTG theorem. This immediately extends  the FTG theorem to  the  sequences   of dependent i.d. random variables that satisfy our sufficient conditions.\\

Clearly, the extremal independence property (\ref{a_e_indep_formula}) is equivalent to
\begin{equation}\label{a_e_indep_formula_for_A}
\left|\P{ \bigcap_{i \in [d]} \overline{ A_i } }-\prod_{i\in [d]} \P{\overline{ A_i } }\right|\to 0,  
\end{equation}
where
the system of events $\A$ is defined by
\begin{align}
\A = \A(n, x) := (A_i)_{i \in [d]},
\quad
A_i := \{ X_i > x \},
\label{event}
\end{align}
and $\overline{ A_i } $ is the complement event of $A_i$. Estimates for the probability of non-occurrence of events appear in many applications in probabilistic combinatorics and number theory.  In particular, to justify  the existence of a certain object, it is sufficient  to show
that the related probability (over all places where this object might appear) is positive;  see, for example, \cite[Section 5]{Alon}.

In this paper, we establish new bounds for \eqref{a_e_indep_formula_for_A} by developing the idea proposed by Galambos \cite{galambos1972distribution, galambos1988variants} and Arratia, Goldstein, Gordon~\cite{AGG}: the weak and strong dependencies between events $(A_i)_{i\in[d]}$ are considered separately, and the bounds do not incorporate the computation of moments of the number of occurrences $Z = \sum_{i \in [d]} \mathbbm{1}(A_i)$ higher than the second one. This allows to overcome the disadvantages of classical bounds. In particular, bounds based on dependency graphs (Lov\'{a}sz Local Lemma (LLL)~\cite{LLL}, Janson's inequality~\cite{janson1988exponential}, Suen's inequality~\cite{suen}) allow complicated dependence structures, but  often fail to characterise the relations quantitatively. Applying  the method of moments  gets complicated when high factorial moments diverge or 
\red{are} hard to compute. Both the Stein-Chen method~\cite{AGG, BHJ1992}  and the method of moments often give suboptimal bounds  in \eqref{a_e_indep_formula_for_A} as they deal with the whole distribution of $Z$ instead of focusing on the probability at $0$. Our bounds for \eqref{a_e_indep_formula_for_A} do not require computation of high moments and  the proofs are based on elementary techniques inspired by LLL.  To demonstrate the simplicity in application and effectiveness of our bounds, we derive new results on  distributions of extremal characteristics of Gaussian systems and   of maximal pattern extensions counts in random network models.\\

The paper is organised as follows.
 Our new bounds for  the extremal independence property (\ref{a_e_indep_formula})
 are stated in Section \ref{Intro_results}   as  Theorem~\ref{Tmain}.
 In Section \ref{Sec:related},  we give a detailed comparison of Theorem~\ref{Tmain} to the related results
 including aforementioned   papers \cite{AGG, galambos1972distribution, galambos1988variants}. In Section \ref{S:bridging}, we give two useful lemmas that facilitate verifying the assumptions.
We prove   Theorem~\ref{Tmain}  in Section~\ref{sec:Janson_bounds}: the upper and lower bounds
are treated separately in  Section \ref{Janson_upper} and Section \ref{Dubicas}, respectively.
Section~\ref{sec:gauss} is devoted to applications of  our new bounds to Gaussian random vectors.
In Section~\ref{rg_applications}, we apply Theorem~\ref{Tmain} for finding the \purple{asymptotic} distribution of maximum number of pattern extensions in binomial random graphs. 

%

%

\section{Sufficient conditions for extremal independence}
\label{Intro_results}

Let $\A := ( A_i )_{i\in[d]}$ be a system of events.  Everywhere below we assume that $\Pr(A_i)\neq 0$. Clearly, this assumption does not lead the loss of the generality since the events of zero probability  can be excluded from $\A$  without affecting the expression in~(\ref{a_e_indep_formula_for_A}).   We represent the dependencies among the events of $\A$ by a graph $\D$
on the vertex set $[d]$ with edges
indicating the pairs of `strongly dependent' events,
while non-adjacent vertices correspond to `weakly dependent' events.
One can think of $\D$ as a set system $(D_i)_{i \in [d]}$,
where $D_i \subseteq [d]$ is the closed neighbourhood of vertex $i$ in graph $\D$.
Moreover, we allow $\D$ to be a directed graph, that is,
there might exist $i, j \in [d]$, such that $i \in D_j$ and $j \not\in D_i$.

To measure the quality of the representation of the dependencies for $\A$
by a graph $\D$,
we introduce the following mixing coefficient:
\begin{align}\label{phi}
\varphi(\A, \D)
:= \max_{i \in [d]}  \left|
\Pr\left(
\bigcup_{j \in [i-1] \setminus D_i   } A_j ~\middle|~ A_i\right)
- \Pr\left(\bigcup_{j \in [i-1] \setminus D_i   } A_j   \right)  \right|.
\end{align}
 This is a special case of $\phi$-mixing coefficient widely used in the probability theory; see, for example, survey \cite{Bradley}.

The influence of `strongly dependent' events is measured by declustering coefficients
$\Delta_1$ and $\Delta_2$ defined by
\begin{align}
\Delta_1(\A, \D)
&:= \sum_{i\in[d]} \P{ A_i\cap \bigcup_{j \in [i - 1] \cap D_i }A_j} \prod_{k\in [d ]\setminus[i]} \P{\overline{ A_k } },
\label{del1} \\
\Delta_2(\A, \D)
&:= \sum_{i\in[d]} \P{ A_i} \P{ \bigcup_{j \in [i - 1] \cap D_i }A_j}
\prod_{k \in [d ]\setminus[i]} \P{\overline{ A_k } }
\label{del2}.
\end{align}
\red{In our model, the choice of graph $\D$ is arbitrary, and therefore the flexibility may leads to the trade-off between the mixing coefficient $\varphi(\A, \D)$
and declustering coefficients $\Delta_1(\A, \D)$ and $\Delta_2(\A, \D)$ for different applications,
since $\Delta_1(\A, \D)$ and $\Delta_2(\A, \D)$
increase as $\D$ gets denser,
and $\varphi(\A, \D)$ typically decreases.}

We are ready to state our sufficient condition for satisfying (\ref{a_e_indep_formula_for_A}).

\begin{theorem}\label{Tmain}
For any system of events $\A = (A_i)_{i \in [d]}$
and graph $\D$ with vertex set $[d]$,
the following bound holds
\begin{align}
\left| \P{ \bigcap_{i \in [d]} \overline{ A_i } }
- \prod_{i\in [d]} \P{\overline{ A_i } } \right|
\le \left(1 - \prod_{i\in [d]} \P{\overline{ A_i } } \right) \varphi +
\max\{ \Delta_1, \Delta_2 \},
\label{eq:Tmain}
\end{align}
where $\varphi = \varphi(\A, \D)$,
$\Delta_1 = \Delta_1(\A, \D)$,
and $\Delta_2 = \Delta_2(\A, \D)$.
\end{theorem}


Although the proof of Theorem~\ref{Tmain} is elementary (see Section~\ref{sec:Janson_bounds}), it gives a very useful and convenient tool to prove   extremal independence property \eqref{a_e_indep_formula}
 stated below.

\begin{corollary}\label{Cmain}
Let $d = d(n) \in \mathbb N$, $\X(n) =(X_1, \ldots, X_d)^T
\in \Reals^d$,
and $\A$ is defined in (\ref{event}).
If for every fixed $x \in \mathbb{R}$,
there is a graph $\D = \D(n, x)$ such that
\begin{align}
\varphi(\A, \D) = o(1),
\quad \Delta_1(\A, \D) = o(1),
\quad \Delta_2(\A, \D) = o(1),
\label{a_e_indep_assumptions}
\end{align}
then, (\ref{a_e_indep_formula}) holds.
\end{corollary}

 Corollary \ref{Cmain}  can be applied to  extremal problems arising in a variety of random models including but not limited to  the following:
\begin{itemize}
		\item random  discrete time vector processes,
		\item random graphs and hypergraphs,
		\item random fields on lattices.
	\end{itemize}
 To illustrate this, we find new sufficient conditions for Gaussian random vectors  to satisfy the
  extremal independence property (\ref{a_e_indep_formula})
  generalising previously  known  conditions; see Theorem~\ref{Gauss}.
  We also extend  Bollob\'{a}s result~\cite{bollobas} on the limit distribution of the maximum degree of
  binomial random graph $\mathcal{G}_{n,p}$ to the hypergraph setting; see Section \ref{max_degree_result}. Our result on the distribution of maximum extension counts implies the law of large numbers by Spencer~\cite{spencer_extensions} and optimizes the denominator for clique extensions;
  see Sections \ref{max_clique_result}--\ref{extensions_further}.  Corollary \ref{Cmain}  simplifies the arguments of  \cite{rodionov2018distribution} for the maximum number of $h$-neighbours  and extends it to unbounded $h$; see  Section \ref{S:max-h-neighbours}.

%
%

There are also a few other straightforward  applications of our new bounds that we decided to cover separately in the future paper(s):
1) distribution of the max number of common neighbours in random regular graphs;
2) distinguishing binomial random graphs by first order logics~\cite{BZ};
3) extensions of  the results   \cite{leadbetterrootzen,  pereiraferreira} on random fields.

Recent results 
\cite{stark2018probability, mousset2020probability, zhang2022asymptotic} 
derive more accurate estimates for $\P{ \bigcap_{i \in [d]} \overline{ A_i } }$
using truncated cumulant series and investigating clusters of dependent random variables. 
It will be interesting to obtain similar extensions of Theorem~\ref{Tmain} 
relying on bounds for clusters of strongly dependent random variables.

\subsection{Related results}
\label{Sec:related}

By the union bound, it is easy to see that
\begin{align*}
\Delta_1(\A, \D)
&\le \Delta_1'(\A, \D)  := \sum_{i\in[d]} \sum_{j \in [i-1] \cap D_i } \P{ A_i \cap A_j }, \\
\Delta_2(\A, \D)
&\le \Delta_2'(\A, \D) := \sum_{i\in[d]} \sum_{j \in [i-1] \cap D_i } \P{ A_i} \P{ A_j}.
\end{align*}
The declustering assumption $\Delta_1'(\A, \D) = o(1)$
is typical in the
study of extremal characteristics of random systems. It guarantees that the clusters of exceedances $A_i$ are negligible.
The assumption $\Delta_2'(\A, \D) = o(1)$ is easy to verify. 
For example, if all probabilities $\Pr(A_i)$ are  of the same order \red{$n^{-1}$},
then this assumption is equivalent to the graph $\D$ to be sparse, which  usually happens in applications. In addition,
 $\Delta_2'(\A, \D)$ can be bounded above by $\Delta_1'(\A, \D) = o(1)$ if the events are monotone.
The most innovative part of Corollary \ref{Cmain} is
the remaining assumption $\varphi(\A, \D) = o(1)$,
which is often easier to check and less restrictive than other mixing assumptions
known in the literature.
The detailed comparisons are given below.

First, we consider a stationary sequence of random variables.
If its distant terms are `weakly dependent',
then we can construct the graph $\D$
by connecting vertices that are close to each other.
Then, omitting some details, the following corresponds to
Leadbetter's mixing condition $D$:
\begin{align}
\left| \P{ \bigcap_{i \in I \cup J} \overline{ A_i } }
- \P{ \bigcap_{i \in I}  \overline{ A_i } } \P{ \bigcap_{i \in J} \overline{ A_i } } \right|
= o(1)
\label{alpha}
\end{align}
for all \red{disjoint} $I, J \subset [d]$ with no edges from $\D$ between them,
see \cite[Eq. \red{(1.2)}]{leadbetterself}.
Although,
(\ref{alpha})
looks similar to our assumption $\varphi(\A, \D) = o(1)$,
none of them not imply the other.
One advantage of our assumption in comparison with (\ref{alpha}) is that
one only needs to check the mixing condition for considerably \red{fewer} pairs of sets $I$ and $J$, namely for $I = [i-1]\setminus D_i$ and $J = \{i\}$ for all $i \in [d]$.
The same conclusion remains valid for the extensions of Leadbetter's mixing condition $D$
for non-stationary sequences and random fields on $\mathbb{Z}^2_+$,
see, H\"usler \cite[Theorem 1.1]{husler} and Pereira, Ferreira {\cite[Proposition 3.2]{pereiraferreira}}, respectively. In fact, our framework is much more flexible
since one can arbitrarily choose the graph $\D$,
without relying on the distances between indices.

Second, we consider the case when $\varphi(\A, \D) = 0$.
For this case, under some additional requirement,
Dubickas \cite[Theorem 1]{Dubickas} proved the following bound:
\begin{align}
\P{ \bigcap_{i \in [d]} \overline{ A_i } }
\ge \prod_{i\in [d]} \P{\overline{ A_i } } - \Delta_2(\A,\D).
\label{dub}
\end{align}
Thus, in this case, (\ref{dub})  gives the lower bound for $\P{ \bigcap_{i \in [d]} \overline{ A_i } }  $
similar to  Theorem \ref{Tmain}.
In the binomial subset setting and under condition $\Delta_1'(\A, \D) = o(1)$,
the matching upper bound for $\P{ \cap_{i \in [d]} \overline{ A_i } }$
can be derived from Janson's inequality \cite{janson1988exponential}.
Our graph-dependent  model is also related to the notions of lopsided (negative) dependency graph \cite{erdHos1991lopsided}
and $\epsilon$-near-positive dependency graph \cite{lu2009new}.
Those are models with one-sided  mixing conditions sufficient for the  lower and upper bounds respectively.

Next,
we compare Corollary \ref{Cmain}
with the results by Galambos \cite{galambos1972distribution, galambos1988variants}.
To our knowledge, he was the first to represent the weak and strong dependences among $(A_i)_{i \in [d]}$ by a graph. Galambos established  the extremal independence property (\ref{a_e_indep_formula})
using the so-called graph-sieve method; see, for example, \cite{galambos1996bonferroni} for detailed overview.
In particular,
Galambos' mixing assumptions require that,
\red{for a fixed graph $\D$,}
\begin{align}
\sum_{S} \left| \P{ \bigcap_{i \in S} A_i } - \prod_{i \in S} \P{ A_i } \right| = o(1),
\label{Gal}
\end{align}
where the sum in (\ref{Gal}) is over all $S \subseteq [d]$ with no edges of $\D$.
Assumption (\ref{Gal}) is very restrictive for many applications
since such set $S$ can be large.
For example, in some of the applications that we consider in Section \ref{rg_applications}, the graph $\D$ is empty so the results in \cite{galambos1972distribution, galambos1988variants}
is of little use, since assumption (\ref{Gal}) is equivalent to
the extremal independence property (\ref{a_e_indep_formula})
that we wish to establish.

To illustrate  the advantage of our approach with respect to the methods  of moments, we briefly consider the following example.  Let $A_i$,  where $i\in [d]$  and $d =  \binom{n}{h}$, to be the event  that the number of common neighbors of the corresponding $h$-subsets of vertices in $\mathcal{G}_{n,p}$ is greater than $a_n+b_n x$  (for some appropriately chosen $a_n,b_n$). In Section~\ref{S:max-h-neighbours}, we show that this system of events obey the asymptotic independence property
 (\ref{a_e_indep_formula_for_A}) despite the fact that the second moment of $Z = \sum_{i \in [d]} \mathbbm{1}(A_i)$   approaches  infinity when $p$ is a sufficiently large constant (depending on $h$).
 In fact, one can get around this difficulty  and modify the random variables so the second moment converges to the desired limit
 by conditioning  on  a certain event $\mathcal{E}_n$ that holds with probability $1-o(1)$. However, it does not help a lot  even for the third moment, and it is not evident that the convergence of the higher moments  can be established directly by a careful choice of random variables.

The \red{aforementioned} difficulty in applying the method of moments was also pointed out  by
Arratia, Goldstein and Gordon in~\cite{AGG}. Based on the Stein-Chen method they discovered  that the computation of two moments is sufficient for Poisson approximation under a certain
mixing condition for weakly dependent random variables.
For the rest of this section, we compare \cite[Theorem 3]{AGG} with our Theorem~\ref{Tmain} as these results have very similar setups.

Arratia et al.~\cite{AGG}  introduced another mixing coefficient different from our $\varphi$:
\begin{align}
	\widetilde\varphi  :=\sum_{i \in [d]} \P{A_i}\sum_{k=0}^d\left|\P{Z^i=k ~\middle|~ A_i}-\P{Z^i=k}\right|\notag,
\end{align}
where  $Z^i = \sum_{j \notin  D_i} \mathbbm{1}(A_i)$.   Their result \cite[Theorem 3]{AGG} states that 
\begin{equation}
	\left| \P{ \bigcap_{i \in [d]} \overline{ A_i } }
	- \prod_{i\in [d]} \P{\overline{ A_i } } \right|
	\leq 2\widetilde\varphi+4\Delta_1''+4\Delta_2''+4\sum_{i\in [d]}\big(\P{A_i}\big)^2,
	\label{eq:Arratia_bound}
\end{equation}
where
$$
\Delta_1''= \sum_{i\in[d]} \sum_{j \in D_i } \P{ A_i \cap A_j }\geq\Delta_1',\quad
\Delta_2''= \sum_{i\in[d]} \sum_{j \in D_i } \P{ A_i} \P{ A_j}\geq\Delta_2'.
$$

To compare $\widetilde{\varphi}$ with our mixing coefficient $\varphi$, we observe that
\begin{equation}
	\widetilde{\varphi} \geq\sum_{i\in [d]}\P{A_i}\cdot \left|\P{\bigcup_{j\notin D_i}A_j ~\middle|~ A_i}-\P{\bigcup_{j\notin D_i}A_j}\right|,
\label{eq:arratia}
\end{equation}
In the typical case
when $\sum_{i\in [d]}\P{A_i}=\Theta(1)$,  $\sum_{i \in [d]}(\P{A_i})^2=o(1)$
 (and up to ordering of vertices in $\D$)
the RHS of (\ref{eq:arratia}) has the same order of magnitude (or even bigger) as $\left(1 - \prod_{i\in [d]} \P{\overline{ A_i } } \right)\varphi$.
Thus, our bound is at least as efficient as \eqref{eq:Arratia_bound}  for such applications. 
\red{Moreover,  the lower bound  (\ref{eq:arratia})  on $\widetilde{\varphi} $ 
could be far from being  sharp, i.e. the actual value of the mixing coefficient $\widetilde{\varphi} $ could be much bigger.}
Furthermore,  Theorem~\ref{Tmain}  surpasses \cite[Theorem 3]{AGG} in several important instances listed below.

\begin{itemize}
	 \item[(1)] \textbf{Slowly decreasing $\sum_{i \in [n]} \big(\P{A_i}\big)^2$.} Clearly, Theorem~\ref{Tmain}  does not have this error term.
	 Thus, our results partially answer the question formulated by Arratia et al.~\cite{AGG} about the extremal independence property \eqref{a_e_indep_formula} in case when Poisson approximation is not good enough.
	
		\item [(2)]  \textbf{Slowly growing $\sum_{i \in [n]}\P{A_i}$.}
	    The term
	    $\left(1 - \prod_{i\in [d]} \P{\overline{A_i} } \right) \varphi$ has additional advantage for upper tail estimates where $\prod_{i\in [d]} \P{\overline{ A_i } }  \rightarrow 1$.

	\item[(3)] \textbf{Inhomogenous random graphs.}  For example, consider the random graph model with vertex set $[n]$ and independent adjacencies, where all adjacencies happen with probability $p$ excluding adjacencies incident to one special vertex. The edges incident to this vertex appear with a slightly higher probability $p'=\frac{a_n+b_nx}{n}=p+(1-o(1))\sqrt{\frac{2p(1-p)\ln n}{n}}$, where $a_n,b_n$ and constant $x\in\mathbb{R}$ are appropriately chosen.  Defining $A_i$ as the event that vertex $i$ in the considered random graph has degree more than $a_n+b_nx$, 
	 our inequality gives the upper bound $O(n^{-1/2})$ in (\ref{eq:Tmain}) while \cite[Theorem 3]{AGG}  gives a useless bound $O(1)$.


	\item[(4)] \textbf{Applications to Gaussian vectors.}  The assumptions of  Theorem~\ref{Tmain} can be verified directly using the Berman inequality;  Section~\ref{sec:gauss}. Combining (\ref{eq:Arratia_bound}) and the Berman inequality directly gives a bound which is  $2^{d-D}$ times  bigger, where $D = \max_{i\in D} |D_i|$. Note that $2^{d-D}$  can be very large if $\D$ is sparse.
	\end{itemize}



\subsection{Bridging sequences}\label{S:bridging}

Here, we state two helpful  lemmas  in applying
Theorem~\ref{Tmain} to study the extremal characteristics of  random combinatorial structures.
It will be convenient to \red{work with} non-scaled random variables $\{X_i\}$.
Everywhere in this section, we assume the following:
\begin{itemize}
	\item  $\X(n) = (X_1,\ldots,X_d)^T \in \Reals^d$ is a sequence of random vectors,   where  $d=d(n) \in \Naturals$;
	\item  $F$ is a continuous cdf on $\mathbb{R}$ and $\mathcal{X}$ is the set of all $x\in\mathbb{R}$ such that $0<F(x)<1$;
	\item there exist
	$a_n$ and $b_n$
	such that
	$\prod_{i=1}^d\P{X_i\leq a_n + b_n x }\to F(x)$
	for any $x \in \mathcal{X}$;
	\item for all $i \in [d]$, denote
	$A_i:=A_i(x)=\{X_i>a_n+b_n x\}$.
\end{itemize}

The first lemma shows that $\varphi(\A,\D)\to 0$ as $n\to\infty$ 
provided that, for all $i\in [d]$
and  $j\in [i-1]\setminus D_i$,
the random variables $X_j$ are approximated by   some random variables $X_j^{(i)}$, which are independent of $X_i$. We will use this lemma to derive the distribution of  the maximum  codegrees in random hypergraphs.

\begin{lemma}\label{L:mixing}
	Let $x\in \mathcal{X}$. Let  sets $D_i  \subseteq [d]\setminus\{i\}$ and  random variables $X_j^{(i)}$ be  such that, for all $j \in [i-1]\setminus D_i$, $X_j^{(i)}$ is independent of $X_i$   and,  for any fixed $\eps>0$,
	\begin{equation}
		\Pr\Big(\max_{ j \in [i-1]\setminus D_i} \left| X_j  - X_j^{(i)} \right|  >  \eps b_n \Big)
		= o(1) \Pr \(A_i\),
		\label{lem:eq:independent}
	\end{equation}
	uniformly over $i \in [d]$. Then $\varphi(\A,\D)\to 0$.
\end{lemma}

The second lemma allows us to transfer the asymptotic distribution of the maximum component
of $\X(n)$  to any  random vector $\Y(n)\in \Reals^d $  that `approximates'  $\X(n)$.  Using this lemma, we will derive the distribution of the maximum clique-extension count  in random graphs from the   results on the maximum degree.

%
%

\begin{lemma}\label{L:transfer}	
	Let $\Y(n)\in \Reals^d $ be a  sequence of random vectors. Assume that, for any $x\in \mathcal{X}$,
	\begin{itemize}	
		\item[(i)] $
		\P { \max_{i\in[d]} X_i\leq a_n + b_n x }\to F(x);
		$
		\item[(ii)] for any fixed $\eps>0$,
		\[
		\Pr(|X_i - Y_i| > \eps b_n)  = o(1)\Pr( X_i> a_n + b_n x ),
		\]
		uniformly over all $i \in [d]$.	
	\end{itemize}
	Then
	%
	$
	\P { \max_{i\in[d]} Y_i\leq a_n + b_n x }\to F(x)
	$
	for all $x\in \mathcal{X}$.
	
\end{lemma}

The  proofs of Lemma \ref{L:mixing} and  Lemma \ref{L:transfer}
require some standard technical calculations, which  we include  in appendix for completeness;
see Sections \ref{S:mixing} and \ref{S:transfer}.


\section{Probability of non-occurrence  of  events}
\label{sec:Janson_bounds}

In this section, we give new lower and upper bounds that allow to make a classification of dependencies between events 
 flexible and that do not require the implication from pairwise to mutual independence. Our bounds are follow-up to the inequalities of Arratia, Goldstein, Gordon~\cite{AGG} and give a certain improvement for applications in various settings (see Section~\ref{Sec:related}). However, the proofs are elementary and inspired by the proof of LLL. Note that our lower bound given in Section~\ref{Janson_lower} is a strict generalisation of Dubickas' inequality~\cite{Dubickas}.

\subsection{Upper bound}
\label{Janson_upper}



Here and in the next section, we use the notations $\Delta_1(\A,\D)$ and $\Delta_2(\A,\D)$ that are defined in (\ref{del1}) and (\ref{del2}) respectively.

\begin{lemma}\label{T:upper}
	Let $\varphi \geq 0$.
	If events $(A_i)_{i\in[d]}$ with non-zero probabilities and
	sets $( D_i \subset [d]\setminus\{i\} )_{i \in [d]}$ satisfy
	\begin{equation}\label{ass:phi}
		\P{ \bigcup_{j \in  [i-1] \setminus D_i} A_j    ~\middle|~ A_i}  - \P{ \bigcup_{j \in [i-1] \setminus D_i} A_j} \le  \varphi,
	\end{equation}
	for all  $i \in [d]$, then
	\begin{align} \label{Upp}
		\P{ \bigcap_{i \in [d]} \overline{ A_i } }  \leq
		\prod_{i\in [d]} \P{\overline{ A_i } }+\varphi\left(1-\prod_{i\in [d]} \P{\overline{ A_i } }\right) +  \Delta_1(\A,\D).
	\end{align}
	
\end{lemma}

\begin{proof}
	Let us prove that, for every $s\in[d]$,
	\begin{align} \label{Upp_induction}
		\P{ \bigcap_{i \in [s]} \overline{ A_i } }  \leq
		(1-\varphi )\prod_{i\in [s]} \P{\overline{ A_i } } +
		\varphi + \sum_{i\in [s]}
		\P{ A_i\cap \bigcup_{j \in [i - 1] \cap D_i }A_j} \prod_{k\in [s]\setminus[i]} \P{\overline{ A_k } }
	\end{align}
	by induction on $s$. The required bound~(\ref{Upp}) is exactly (\ref{Upp_induction}) when $s=d$.\\
	
	For $s=1$, (\ref{Upp_induction}) follows from $\varphi\geq 0$.  Assume that (\ref{Upp_induction}) holds for some $s\in[d-1]$.  Let
	\begin{equation}
		B:= \bigcup_{j \in [s]\setminus D_{s+1} } A_j,
		\qquad C:= \bigcup_{j \in [s] \cap D_{s+1} }A_j.
		\label{def:B_and_C}
	\end{equation}
	Note that
	\begin{equation}\label{eq:simple}
		1 -  \P{\overline{ A_{s+1} } \m \bigcap_{i \in [s]} \overline{ A_i } }
		=  \P{A_{s+1} \m \overline B \cap \overline C}
		\geq \P{A_{s+1} ~\middle|~ \overline{B}}  (1 - \P{C ~\middle|~  A_{s+1}\cap \overline{B}}).
	\end{equation}
	By \eqref{ass:phi}, we have $\P{ \overline{B} \m A_{s+1}} \geq  \P{ \overline{B}} - \varphi$. Therefore,
	\[
	\P{A_{s+1}\m \overline{B}}
	= 	\frac{\P{ \overline{B} \m A_{s+1}}}{\Pr(\overline{B})}    \Pr(A_{s+1})
	\geq  \left(1 -  \frac{\varphi}{\P{\overline{B}}}\right) \Pr(A_{s+1}).
	\]
	We also find that
	\[
	\P{ C\m A_{s+1}  \cap  \overline{B}} = \frac{	 \P{ C\cap \overline B ~\middle|~ A_{s+1}  } }{
		\P{ \overline{B}  ~\middle|~ A_{s+1}}}
	\le   \frac{ \P{ C \m A_{s+1}} }{ \P{\overline{B}} - \varphi}.
	\]
	Using the above two bounds  in \eqref{eq:simple}, we  derive that
	\begin{align*}
		\P{\overline{ A_{s+1} } \m \bigcap_{i \in [s]} \overline{ A_i } } \leq
		1 -   \left(1 -  \frac{\varphi}{\P{\overline{B}}}\right) \Pr(A_{s+1}) + \frac{ \P{ A_{s+1} \cap C} }{ \P{\overline{B}}}.
	\end{align*}
	Then, since  $\P{\overline{B}} \geq  \P{\bigcap_{i \in [s]} \overline{ A_i }}  $, we get
	\[
	\P{\bigcap_{i \in [s+1]} \overline{ A_i }}
	\leq    \P{\bigcap_{i \in [s]} \overline{ A_i }} \P{\overline{A_{s+1}}} + \varphi \Pr(A_{s+1}) + \Pr(A_{s+1}\cap C).
	\]
	By (\ref{Upp_induction}), we have
	\begin{align*}
		\P{\bigcap_{i \in [s+1]} \overline{ A_i }}
		&\le (1-\varphi )\prod_{i\in [s+1]} \P{\overline{ A_i } } + \varphi  \\
		&\quad+ \sum_{i\in [s+1]} \P{ A_i\cap \bigcup_{j \in [i - 1] \cap D_i }A_j}
		\prod_{k\in [s+1]\setminus[i]} \P{\overline{ A_k } }.
	\end{align*}
	This completes the proof.
\end{proof}

\subsection{Lower bound}
\label{Janson_lower}

\begin{lemma}[Generalised Dubickas' inequality]\label{Dubicas}
Let $\varphi \geq 0$.
If events $(A_i)_{i\in[d]}$ with non-zero probabilities and
sets $D_i \subset [d]\setminus\{i\}$ satisfy
\begin{equation}\label{ass:phi2}
\P{ \bigcup_{j \in [i-1] \setminus D_i} A_j}
- \P{ \bigcup_{j \in  [i-1] \setminus D_i} A_j    ~\middle|~ A_i}
\le  \varphi,
\end{equation}
for all  $i \in [d]$, then
\begin{align}
\P{ \bigcap_{i \in [d]} \overline{A_i} }
&\ge
\prod_{i \in [d]} \P{ \overline{ A_i } }-\varphi\left(1-\prod_{i \in [d]} \P{ \overline{ A_i } }\right) - \Delta_2(\A,\D).
\label{lowerB1}
\end{align}
\label{Tlower}
\end{lemma}

\begin{proof}

Let us prove that, for every $s\in[d]$,
\begin{equation} \label{Lower_induction}
\P{ \bigcap_{i \in [s]} \overline{A_i} }\ge (1 + \varphi) \prod_{i \in [s]} \P{ \overline{ A_i } }
- \varphi-\sum_{i\in[s]} \P{ A_i} \P{ \bigcup_{j \in [i - 1] \cap D_i }A_j}
\prod_{k \in [s]\setminus[i]} \P{\overline{A_k} }
\end{equation}
by induction on $s$. The required bound~(\ref{lowerB1}) is exactly (\ref{Lower_induction}) when $s=d$.

For $s=1$, (\ref{Lower_induction}) is straightforward since $\varphi\geq 0$. Assume that (\ref{Lower_induction}) holds for $s\in[d-1]$. Consider the events $B$ and $C$ defined in (\ref{def:B_and_C}). Then
\begin{align*}
\P{ \overline{A_{s+1}} \m \bigcap_{i \in [s]} \overline{A_i} }
= 1 - \frac{ \P{ A_{s+1} \cap \overline{B} \cap \overline{C} } }{ \P{\overline{B} \cap \overline{C}} }
\ge 1 - \frac{ \P{ \overline{B} \m A_{s+1} }  }{ \P{\overline{B} \cap \overline{C}} } \P{ A_{s+1} }.
\end{align*}
From (\ref{ass:phi2}), we have $\Pr(\overline{B}|A_{s+1})\leq\Pr(\overline{B})+\varphi$. Therefore,
\begin{align}
\P{ \overline{A_{s+1}} \m \bigcap_{i \in [s]} \overline{A_i} }
\ge 1 - \frac{ \P{ \overline{B} } + \varphi  }{ \P{\overline{B} \cap \overline{C}} } \P{ A_{s+1} }.
\label{opp1}
\end{align}
Moreover,
\begin{align}
\P{\overline{ B } }
= \P{\overline{ B } \cap \overline{ C } } + \P{\overline{ B } \cap C}
& \le \P{ \bigcap_{i \in [s]} \overline{A_i} } + \P{C}.
\label{opp2}
\end{align}
Combining (\ref{Lower_induction}),~(\ref{opp1})~and~(\ref{opp2}), we get
\begin{align*}
\P{\bigcap_{i \in [s+1]} \overline{A_i}}
&\ge \P{\bigcap_{i \in [s]} \overline{A_i}} \P{\overline{A_{s+1}}} - \varphi \Pr(A_{s+1}) -
\P{ \bigcup_{j \in [s]\cap D_{s+1}}A_j } \Pr(A_{s+1}) \\
&\ge (1 + \varphi) \prod_{i \in [s+1]} \P{ \overline{ A_i } } - \varphi
- \sum_{i\in[s+1]} \P{ A_i} \P{ \bigcup_{j \in [i - 1] \cap D_i }A_j}
\prod_{k \in [s+1]\setminus[i]} \P{\overline{A_k} }.
\end{align*}
This completes the proof.
\end{proof}

\begin{remark}
	As mentioned in Section \ref{Sec:related},
the special case of  (\ref{lowerB1}) with $\varphi = 0$
proves Dubickas' inequality (\ref{dub}).
Note also that  our condition
\[\P{ \bigcup_{j \in [i-1] \setminus D_i} A_j}
\le \P{ \bigcup_{j \in  [i-1] \setminus D_i} A_j ~\middle|~ A_i},
\]
  is weaker than the Dubickas' requirement on the connection between pairwise and mutual independencies.

\end{remark}


%
%

\section{Applications in Gaussian systems}
\label{sec:gauss}

In this section we assume that $\X(n)$ is a Gaussian vector for all $n\in \mathbb{N}.$ Our main purpose is to provide conditions under which this system satisfies the assumptions of Theorem \ref{Tmain}.

Assume first that \red{under some linear normalisation,
$\X(n)$} is the first $n$ random variables of a given sequence $\{X_n\}_{n\ge 1}$. In i.i.d. case, it is well-known that the distribution function of the maximum of this sequence under a specific linear normalisation tends to a standard Gumbel law, that is, for all $x\in \mathbb{R}$,
\begin{equation}\P{\max_{i\in [n]}X_i \le a_n + b_n x} \to e^{-e^{-x}} \label{gumbel_normal}\end{equation}
for some non-random sequences $\{a_n\}$ with $a_n>0, n\in \mathbb{N},$ and $\{b_n\}.$ 
But \red{does} the relation \eqref{gumbel_normal} hold if $\{X_n\}_{n\ge 1}$ are not independent or/and identically distributed? Berman \cite{berman} showed that \eqref{gumbel_normal} remains true for stationary Gaussian sequence $\{X_n\}_{n\ge 1}$ under the following remarkable condition
\begin{equation} r(n) \log n \to 0,\label{origin_berman}\end{equation} where $r(n)$ is a covariance function of $\{X_n\}_{n\ge 1}.$ It turns out, that the Berman condition \eqref{origin_berman} is necessary and sufficient in some sense for Gaussian stationary sequence $\{X_n\}_{n\ge 1}$ to satisfy \eqref{gumbel_normal} with the same normalising sequences $\{a_n\}$ and $\{b_n\}$ as in i.i.d. case. Indeed, Mittal and Ylvisaker \cite{mittal} discovered that if $r(n) \log n \to \gamma,$ then the probability in \eqref{gumbel_normal} converges to a convolution of the standard Gumbel and some Gaussian distribution, thus the property (\ref{a_e_indep_formula}) does not hold in this case.

Next, H\"usler \cite{husler1} found the conditions under which the relation \eqref{a_e_indep_formula} is fulfilled for non-stationary Gaussian sequence $\{X_n\}_{n\ge 1}.$ One of the conditions imposed by H\"usler was the modified Berman condition \[\sup_{|i-j|>n} \rho(i,j) \log n \to 0,\] where $\rho(i,j)$ is a correlation function of the sequence $\{X_n\}_{n\ge 1}.$ The property (\ref{a_e_indep_formula}) for non-stationary Gaussian random fields on $\mathbb{Z}^2_+$ with mean zero and unit variance was treated in \cite{Choi} and \cite{pereira}. In \cite{jakubowski}, for stationary Gaussian random field on $\mathbb{Z}^k$, the property (\ref{a_e_indep_formula}) was proved under some multivariate modification of the Berman condition \eqref{origin_berman}. It is interesting that violation of this Berman condition for stationary Gaussian random field in at least one direction in $\mathbb{Z}^k$ entails violation of (\ref{a_e_indep_formula}), see \cite{jakubowski+}.\\

Let us now switch to our {\it most general} case when the $i$-th component of $\X(n)$ depends on $n$. Assume that $d=d(n)$ and $\prod_{i=1}^d\P{X_i \le a_n + b_n x}\to F(x)>0$ for any fixed $x\in\mathbb{R}.$ For every $n\in\mathbb{N}$ and $i,j\in[d]$, set $r_{ij}=r_{ij}(n) = \frac{{\rm cov}(X_i,X_j)}{\sqrt{{\rm Var} X_i{\rm Var} X_j}}$.
Denote $u_i=u_i(n)=\frac{a_n + b_n x-\E{X_i}}{\sqrt{{\rm Var} X_i}}$ and $u_{\min}(n)=\min_{i\in[d]}u_i$.

\begin{theorem}\label{Gauss}
\red{Assume that} for every $x\in \mathbb{R}$, 
there is a graph $\D=\D(n,x)$ such that the Gaussian system $\X(n)$ satisfies the following conditions.
\begin{itemize}
\item[(G1)] $\liminf_n u_{\min}(n) > 1.$
\item[(G2)] $\max_{i\in [d]} \max_{j\in [i-1]\backslash D_i} |r_{ij}| \log d\to 0$.
\item[(G3)] $\sum_{i\in [d]}\sum_{j\in [i-1]\cap D_i} \exp\Big(-\frac{u_{i}^2 + u_{j}^2}{2(1 + |r_{ij}|)}\Big) \to 0$.
\item [(G4)]  $\limsup_n\max_{i\ne j} |r_{ij}|< \rho$ for some fixed
$\rho \in (0,1)$.
\end{itemize}
Then $\P{\max_{i\in [d]}X_i \le a_n + b_n x} \to F(x)$.\\
\end{theorem}

We prove Theorem~\ref{Gauss} later in this section, but, first, we compare  it with the previously known results.  \\


Theorem \ref{Gauss} implies the results of \cite{berman},  \cite{jakubowski} and \cite{pereira} mentioned above.  Indeed, the Berman-type condition (G2) is more general than the corresponding conditions in these works. The condition (G4) is fulfilled for stationary sequence in \cite{berman} and stationary field in \cite{jakubowski} and is assumed in \cite{pereira}. At last, the specific choice of the sets $\{D_i\}$ (all of them should have the same form and size) with application of the Berman-type condition guarantees the fulfillment of (G3).  It  is also straightforward to derive  \cite[Theorem 4.1]{husler1}  from our Theorem~\ref{Gauss} if $u_{\min} = \Omega(\sqrt{\log d}).$

Next,  we compare the conditions (G1)--(G4) with the conditions in the above mentioned result of H\"usler \cite{husler1}.

\begin{itemize}
\item In contrast to \cite{husler1}, we do not require that $u_{\min}(n)$ tends to the endpoint of the limit distribution function $F$ (which is infinity for the Gumbel distribution) 
\red{but only need} the weaker condition (G1).

\item (G2) is a Berman-type condition 
\red{whose analogue} was also used by H\"usler. Next, it is easy to see that the condition (G3) follows from a more convenient assumption
$$
\exp\Big(-\frac{(u_{\min}(n))^2}{1 + \rho}\Big)\sum_{i\in[d]} |[i-1] \cap D_i|\to 0.
$$
This condition (together with (G2)) is more flexible than the conditions being imposed on Gaussian sequences and fields in the literature. The sets of indices $\{D_i\}$ may not be intervals in \red{the} one-dimensional case, in contrast to \cite{berman}, \cite{husler1}, 
and \red{may be} neither parallelepipeds nor spheres in the multi-dimensional case, in contrast to the choice of corresponding sets in \cite{pereira} and \cite{jakubowski} respectively. Moreover, the form of $|D_i|$ can strongly depend on $i.$

\item Finally, (G4) was also assumed by H\"usler.
\end{itemize}

\subsection{Proof of Theorem~\ref{Gauss}}
Set $\widetilde X_i=\frac{X_i-\E{X_i}}{\sqrt{\mathrm{Var} X_i}}.$ Let us check that the conditions of Corollary \ref{Cmain} 
\red{hold} for \[A_i=\left\{\frac{X_i - a_n}{b_n}>x\right\} = \{\widetilde X_i>u_i\},\quad i\in[d].\] This will immediately give the statement of Theorem~\ref{Gauss}.\\

First, recall the well-known relation for standard normal $\eta$ (see, for example, \cite[Proposition 2.4.1]{piterbarg})

\begin{equation}
\frac{1}{\sqrt{2\pi} x} e^{-x^2/2} (1 - x^{-2}) \le \P{\eta>x} \le \frac{1}{\sqrt{2\pi} x} e^{-x^2/2}, \quad x>0.\label{normal_bounds}
\end{equation}
Here and in what follows, $C$
denotes a positive constant whose value is large enough (it may be different in different places --- we use the same notation to avoid introducing many letters or indices).
Using the upper bound from \eqref{normal_bounds}, we observe that
\begin{equation} \label{for_A3}
\P{A_i}\P{A_j} \le C \frac{1}{u_{i}u_{j}} \exp\left(-\frac{u_{i}^2+u_{j}^2}{2}\right) \le C\exp\left(-\frac{u_{i}^2+u_{j}^2}{2(1 + |r_{ij}|)}\right).
\end{equation} Therefore, (G3) implies that the assumption $\Delta^\prime_2(\A, \D) = o(1)$ and hence the assumption $\Delta_2(\A, \D) = o(1)$ are fulfilled.\\

Now, we have the following trivial bound for the sum in $\Delta^\prime_1(\A, \D)$
\[
\sum_{i\in [d]}\sum_{j\in[i-1]\cap D_i} \P{A_i\cap A_j} \le \sum_{i\in [d]}\sum_{j\in[i-1]\cap D_i} \Big(\P{A_i}\P{A_j} + d(i,j)\Big),
\]
where
\[d(i,j) = \left|\P{\widetilde X_i \le u_{i}, \widetilde X_j \le u_{j}} - \P{\widetilde X_i \le u_{i}} \P{\widetilde X_j \le u_{j}}\right|.\] The latter follows from the relation
\begin{equation}
|\P{A\cap B} - \P{A}\P{B}| = |\P{\overline{A}\cap \overline{B}} - \P{\overline{A}} \P{\overline{B}}|.
\label{from_complements_to_events}
\end{equation}
Direct application of the famous Berman inequality (cf. Theorem 4.2.1 \cite{leadbetter_book}) gives us
\begin{equation} d(i,j) \le C\frac{|r_{ij}|}{1 - r_{ij}^2} \exp\left(-\frac{u_{i}^2 + u_{j}^2}{2(1 + |r_{ij}|)}\right). \label{d(i,j)} \end{equation}
Therefore, we easily obtain by (G4), \eqref{for_A3} and \eqref{d(i,j)}
\begin{eqnarray*}
\sum_{i\in [d]}\sum_{j\in[i-1]\cap D_i} \Big(\P{A_i}\P{A_j} + d(i,j)\Big) & \le & C \sum_{i\in [d]}\sum_{j\in[i-1]\cap D_i} \exp\left(-\frac{u_{i}^2+u_{j}^2}{2(1 + |r_{ij}|)}\right),
\end{eqnarray*}
and the right-hand side vanishes by (G3). Thus, we verified the assumption $\Delta^\prime_1(\A, \D)=o(1)$ and hence the assumption $\Delta_1(\A, \D)=o(1);$ it remains to justify $\varphi(\A,\D)=o(1)$.\\

By definition \eqref{phi} and using \eqref{from_complements_to_events} again, we have
\begin{eqnarray*}\nonumber \varphi(\A, \D) &=& \max_{i\in[d]} \frac{1}{\P{A_i}}\Big|\P{\cup_{j\in [i-1]\backslash D_i}A_j \cap A_i} - \P{\cup_{j\in [i-1]\backslash D_i}A_j}\P{A_i}\Big|\\
& = & \max_{i\in [d]}\frac{1}{\P{A_i}}\Big|\P{\cap_{j\in [i-1]\backslash D_i}\overline{A_j} \cap \overline{A_i}} - \P{\cap_{j\in [i-1]\backslash D_i}\overline{A_j}}\P{\overline{A_i}}\Big| \label{first_relation}
\end{eqnarray*}
Applying the Berman inequality again, we get
\begin{equation*}
 \left|\P{\bigcap\limits_{j\in [i-1]\backslash D_i}\overline{A_j} \cap \overline{A_i}} - \P{\bigcap\limits_{j\in [i-1]\backslash D_i}\overline{A_j}}\P{\overline{A_i}}\right| \le C \sum_{j\in [i-1]\backslash D_i}\frac{|r_{ij}|}{1 - r_{ij}^2} \exp\left(-\frac{u_{i}^2 + u_{j}^2}{2(1 + |r_{ij}|)}\right).
\end{equation*}
Thus, by (G1), (G2), (G4) and \eqref{normal_bounds}, we obtain
\begin{eqnarray} \nonumber
\varphi(\A, \D)
& \le & C \max_{i\in [d]} \left(\frac{u_{i}^3}{u_{i}^2-1}e^{u^2_{i}/2} \sum_{j\in [i-1]\backslash D_i} \frac{|r_{ij}|}{1 - r_{ij}^2} \exp\left(-\frac{u_{i}^2 + u_{j}^2}{2(1 + |r_{ij}|)}\right)\right)\\
\nonumber
&\le&C \max_{i\in [d]} \left(u_{i}e^{u^2_{i}/2} \sum_{j\in [i-1]\backslash D_i}|r_{ij}|\exp\left(-(1 - |r_{ij}|)\frac{u_{i}^2 + u_{j}^2}{2}\right)\right) \\
&=& C \max_{i\in[d]}\left(u_{i} \sum_{j\in [i-1]\backslash D_i}|r_{ij}|	\exp\left(-\frac{u_{j}^2}{2} + |r_{ij}|\frac{u_{i}^2+ u_{j}^2}{2}\right)\right). \label{for_A4}
\end{eqnarray}

\red{ First let us assume that} $u_{i}/\sqrt{\log d} > 2\sqrt{2}+\delta$ for some $\delta>0,$ then $\exp(-u_{i}^2/4) = o(d^{-2}).$
Therefore,
$$
\sum_{j\in [i-1]} \exp\Big(-\frac{u_{i}^2 + u_{j}^2}{2(1 + |r_{ij}|)}\Big)\leq
\sum_{j\in [i-1]} \exp\Big(-\frac{u_{i}^2}{4}\Big) = o(d^{-1}),
$$
and the sum in the left-hand side does not affect the asymptotic of the double sum in (G3).
\red{Let us  change the graph $\D$ so that} $D_i = [d]$ and derive that
\[\Big|\P{\cup_{j\in [i-1]\backslash D_i}A_j \cap A_i} - \P{\cup_{j\in [i-1]\backslash D_i}A_j}\P{A_i}\Big|=0.\]

\red{Now let us instead assume that} $u_{i}/\sqrt{\log d}\leq 2\sqrt{2}+o(1)$. If $u_j/\sqrt{\log d}\geq \frac{3(1+4\rho)}{1-\rho}$  for $j\in [i-1]\backslash D_i$, then
$$
\exp\left(-\frac{u_{j}^2}{2} + |r_{ij}|\frac{u_{i}^2+ u_{j}^2}{2}\right)=o\left(\frac{1}{d\log d}\right).
$$
Therefore, we may assume that
\begin{equation}
\max_{j\in [i-1]\backslash D_i} u_{j} \le \frac{3(1+4\rho)}{1-\rho} \sqrt{\log d}.
\label{max_u}
\end{equation}
Hence, by (G2), \[|r_{ij}| \frac{u_{i}^2+u_{j}^2}{2} = o(1)\] uniformly over $i\in[d]$ and $j\in [i-1]\backslash D_i$. Using the latter, (G1), (G2), \eqref{normal_bounds}, and \eqref{max_u}, we derive for the right-hand side of \eqref{for_A4}
\[ \max_{i\in[d]} \sum_{j\in [i-1]\backslash D_i} u_{i}|r_{ij}| \exp\Big(-\frac{u_{j}^2}{2} + |r_{ij}|\frac{u_{i}^2 + u_{j}^2}{2}\Big)  =  \max_{i\in [d]}\sum_{j\in [i-1]\backslash D_i}  u_{i}|r_{ij}| \exp\Big(-\frac{u_{j}^2}{2} + o(1)\Big)\]
\[ \le C \max_{i\in [d]}\sum_{j\in [i-1]\backslash D_i} u_{i} u_{j}|r_{ij}| \P{A_j}
 \le  C  \max_{i\in [d]}\Big(\max_{j\in [i-1]\backslash D_i} |r_{ij}|\log d \sum_{j\in [i-1]\backslash D_i} \P{A_j} \Big)\to 0,\]
where the last relation holds since
$$
\sum_{j\in[d]}\P{A_j}\leq-\sum_{j\in[d]}\log(1-\P{A_j})\to-\log F(x)<\infty.
$$
The result follows.\quad $\Box$\\

\section{Applications in random graphs}
\label{rg_applications}



Let us recall that $\mathcal{G}_{n,p}$ is a random graph on the vertex set $[n]=\{1,\ldots,n\}$  distributed as
\[
 \Pr(\mathcal{G}_{n,p}=G)=p^{e(G)}(1-p)^{{n\choose 2}-e(G)},
 \]
 where $e(G)$ is the number of edges of a graph $G$  with the vertex set $[n]$  (i.e., every pair of distinct vertices of $[n]$ is adjacent with probability $p$ independently of all the others).

In~\cite{bollobas}, Bollob\'{a}s proved that, for $p=\mathrm{const}$, the maximum degree $\varDelta$ of $\mathcal{G}_{n,p}$ after appropriate rescaling converges to Gumbel distribution. More formally, there exist sequences $a_n$ and $b_n$ (the exact values are known) such that $\frac{\varDelta-a_n}{b_n}$ converges in distribution to a standard Gumbel random variable. Ivchenko proved~\cite{ivchenko} that the same holds for $p$ such that $p(1-p)\gg\frac{\log n}{n}$. In other words, for the rescaled degree sequence of $\mathcal{G}_{n,p}$, the extremal independence property~(\ref{a_e_indep_formula}) holds. This is not unexpected since the dependence of degrees of two vertices of the random graph is `focused' in the only edge between these vertices. 
In Section~\ref{max_degree_result}, we show that Theorem~\ref{Tmain} implies the same result for the maximum degree of
binomial random {\it hypergraph} that can not be obtained by the approach of Bollob\'{a}s and Ivchenko directly.

The results of Bollob\'{a}s and Ivchenko can be viewed as a particular case of the following problem suggested by Spencer in~\cite{spencer_extensions}. Let $G$ be a graph, and $H$ be its subgraph on $h$ vertices. Define $d(H,G)=\frac{|E(G)|-|E(H)|}{|V(G)|-|V(H)|}$ (here, as usual, $V(G)$ and $E(G)$ are the set of vertices and the set of edges of $G$ respectively). Let the pair $(H,G)$ be {\it strictly balanced} and {\it grounded} i.e.,
\begin{itemize}
\item for every $S$ such that $H\subsetneq S\subsetneq G$, \red{we have} $d(H,S)<d(H,G)$, \red{and}

\item there is an edge between $V(H)$ and $V(G)\setminus V(H)$ in $G$.
\end{itemize}

For brevity, we \red{denote by} $[n]_h$ and ${[n]\choose h}$ the set of all $h$-tuples of distinct vertices from $[n]$ and the set of all $h$-subsets of $[n],$ respectively. 
For an $h$-tuple $\mathbf{x}=(x_1,\ldots,x_h)\in[n]_h$, \red{denote by} $X_{\mathbf{x}},$ the number of {\it $(H,G)$-extensions} of $\mathbf{x}$ in $\mathcal{G}_{n,p}$ (i.e., the number of copies of $(V(G),E(G)\setminus E(H))$ in $\mathcal{G}_{n,p}$ in which each vertex $v_j$, $j\in[h]$, of $H$ maps onto $x_j$). For example, the degree of a vertex $u$ equals $X_u$ when $h=1$ and $G=K_2$ (as usual, we denote by $K_r$ a complete graph on $r$ vertices and call it an {\it $r$-clique}). Spencer raised the question about the deviation of $X_{\mathbf{x}}$ from its expectation and proved that
\begin{equation}
\frac{\max_{\mathbf{x}\in[n]_h}|X_{\mathbf{x}}-\mu| }{\mu}\stackrel{\Pr}\rightarrow 0
\label{extensions_lln}
\end{equation}
whenever $\mu:=\E{X_{(1,\ldots,h)}}=\Theta\left(n^{|V(G)|-|V(H)|}p^{|E(G)|-|E(H)|}\right)\gg\log n$. In Section~\ref{max_clique_result}, we show that Theorem~\ref{Tmain} results in a tight lower bound of a possible denominator in the law of large numbers~(\ref{extensions_lln}) for a slightly more narrow range of $p$ and some specific strictly balanced and grounded $(H,G)$. More precisely, for $h=1$ and $G$ being a clique (its size may depend on $n$), we prove that $\max_{u\in[n]}X_u$ after appropriate rescaling converges to Gumbel distribution. Moreover, as we discuss in Sections~\ref{S:max-h-neighbours}~and~\ref{extensions_further}, these techniques can be applied for $h>1$ as well.

\subsection{Maximum degree and codegree  in hypergraphs}
\label{max_degree_result}

Let $\HG_{n, k, p}$ be the $k$-uniform binomial random hypergraph with the vertex set $[n]$.
Recall that  every $k$-set from ${[n] \choose k}$ appears as an edge in $\HG_{n, k, p}$  with probability $p$ independently.
For a set $S\subseteq[n]$ with $|S|< k$ let
$X_S$ be  the codegree of
 $S$  in $\HG_{n, k, p}$ (i.e., the number of edges of $\HG_{n, k, p}$  containing $S$). In particular,  $X_i$ is the degree of a vertex $i$. Note that
\begin{equation}\label{D-dist}
	X_S \sim \Bin\left(\binom{n-|S|}{k-|S|}, p\right).
\end{equation}
In this section, using Theorem \ref{Tmain},
we show that, 
\red{under some assumptions on the parameters $k$ and $p$ (in terms of $n$)},
the asymptotic distribution of  $\max_S X_S$   is the same as if the  variables $X_S$ were independent. For independent random variables, the asymptotic distribution is given by
the following lemma.
\begin{lemma}\label{L:bin}
Let   $d=d(n)\in \Naturals$, $N=N(n) \in \Naturals$, and
$p=p(n)\in (0,1)$   satisfy
\[
	     Np(1-p) \gg \log^3 d \gg 1.  
\]
If $\xi_1,\ldots,\xi_d$ are $\Bin(N,p)$ independent random variables, then $\left[\max_{i\in[d]}\xi_i-a_n\right]/b_n$ converges in distribution to a standard Gumbel random variable with  $a_n$ and $b_n$ defined by
 \begin{equation}\label{def:ab}
 \begin{aligned}
 	 a_n&= a_n(d, N, p):= pN +\sqrt{2 Np(1-p)  \log d}\left(1 - \frac{\log\log d}{4\log d} - \frac{\log( 2 \sqrt \pi)}{2\log d}\right),
 \\
 	b_n &= b_n(d, N,p):= \sqrt{\frac{ Np(1-p) }{2 \log d} }.
 \end{aligned}
 \end{equation}
%
 \end{lemma}
 \begin{proof}
  For $p$ bounded away from $0$ and $1$ we refer to   \cite[Theorem 3]{nadarajah2002}.   
   For $p \rightarrow 0$, $p\gg\frac{\log^3 d}{N}$,  we find by   \cite[Lemmas 4 and 5]{ivchenko} that
   \begin{equation}\label{eq:intermediate}
   		d \Pr(\xi_1 > a_n + b_nx) \rightarrow e^{-x}.
   \end{equation}
    Since $d\rightarrow \infty$, 
    \[
    		\Pr\left(\max_{i\in [d]} \xi_i \leq a_n + b_n x \right)
    		=  \big( \Pr(\xi_1  \leq a_n + b_nx)\big)^d
    		= \left(1- \frac{e^{-x}+o(1)}{d}\right)^d \rightarrow e^{-e^{-x}}.
    \]
    Finally, if $\frac{\log^3 d}{N}\ll 1-p=o(1)$, then \eqref{eq:intermediate} can be obtained similarly    by applying  de Moivre--Laplace theorem (see, e.g., \cite[Theorem 1.6]{bollobas2001random}). 
 \end{proof}

\begin{remark}
In fact,   Lemma  \ref{L:bin} can be extended to  the range $Np(1-p) \gg \log d$.
This  would  involve a more complicated expression for $a_n$, while  $b_n$ remains the same,  see  \cite[Lemma 5]{ivchenko}. However, such a generalisation is not needed for the applications we consider.
\end{remark}

In the next theorem, we show that, 
\red{under certain assumptions},
the maximum degree  in     the random hypergraph  $\HG_{n, k, p}$ converges to the Gumbel distribution.

\begin{theorem}\label{max_degree_hyper}
Assume $p=p(n)\in(0,1)$ and $k=k(n)\in\{2,\ldots,n\}$ are such that
\begin{equation}  \binom{n-1}{k-1}p(1-p) \gg \log^3 n, \qquad k  \ll n/\log^2 n. \label{cond_le3.1}
\end{equation}
Then $\left[\max_{i\in[n]}X_i-a_n\right]/b_n$ converges in distribution to a standard Gumbel random variable, where
$a_n = a_n\left(n,  \binom{n-1}{k-1} ,p\right)$ and $b_n = b_n\left(n,  \binom{n-1}{k-1} ,p\right)$ are
defined in \eqref{def:ab}.

%
%
\end{theorem}

\begin{proof}
Take any $x\in \Reals$. For all $i\in[n]$, let $A_i:= \{X_i > a_n + xb_n \}$.
Let $d:=n$  and $N :=\binom{n-1}{k-1}$. By  
Lemma \ref{L:bin},  we find that
\begin{equation}\label{eq:F}
\prod_{i \in [n]} \Pr( \overline{A}_i)\rightarrow  e^{-e^{-x}}.
\end{equation}
For $i\in[d]$, let $D_i=\varnothing$. Then $\Delta_1(\A, \D) = \Delta_2(\A, \D)=0$.
By \purple{Corollary~\ref{Cmain}},  we  only need to show that
$\varphi(\A,\D)=o(1)$.
We employ Lemma \ref{L:mixing}   to verify  it.
 Note that $F(x):= e^{-e^{-x}}$ is the cdf of the standard Gumbel distribution. In particular, $F$ is continuous and $0<F(x)<1$ for all $x\in\Reals$. To \red{apply} Lemma \ref{L:mixing}, it remains to construct random variables $X_j^{(i)}.$ For $j\in[d]\setminus i$, define $X_j^{(i)}:= \E{X_j ~\middle|~ \obj_i}$,  where $\obj_i$ is the set of edges of $\HG_{n, k, p}$ that does not contain the vertex $i$. Clearly,
$X_j^{(i)}$ is independent of $X_i$ because the  random set $\obj_i$ is independent of  $X_i$. Recalling that $X_{i,j}:=X_{\{i,j\}} \sim \Bin\left( \binom{n-2}{k-2}, p\right)$ is the number of edges of $\HG_{n, k, p}$  containing both $i$ and $j$, we get
\begin{equation}\label{eq:X-X}
	 X_j - X_j^{(i)} =  X_{i,j} - \E{X_{i,j} ~\middle|~ \obj_i} =  X_{i,j} - \E{X_{i,j}}.
\end{equation}

Next, we estimate the probability that $|X_{i,j} - \E{X_{i,j}}| > \eps b_n$.  Here, without loss of the generality, we may assume that $p\leq \frac12$. Otherwise, we can consider the random variable
$\binom{n-2}{k-2} - X_{i,j} \sim\Bin\left( \binom{n-2}{k-2}, 1-p\right)$ and repeat the arguments. By the assumptions, we get that $ b_n =  \sqrt{\dfrac{\binom{n-1}{k-1} p(1-p) }{2 \log n} }$ satisfies
\begin{align*}
	  b_n \gg \log n \qquad \text{and} \qquad   \frac{b_n^2}{\E{X_{i,j}}} =
	  \frac{\binom{n-1}{k-1}(1-p)}{ 2 \binom{n-2}{k-2} \log n} \gg \log n.
\end{align*}
Applying the Chernoff bound 
\red{(see, for example, \cite[Theorem 2.1]{JLR})}, we find that, for any fixed $\eps>0$,
\begin{equation}\label{D-conc}
	\Pr\left(|X_{i,j} - \E{X_{i,j}}| > \eps b_n\right)
	\leq 2 \exp \left(- \frac{(\eps b_n)^2}{2 \E{X_{i,j}} +\eps b_n}\right)
	 = e^{-\omega(\log n)}.
\end{equation}

Combining \eqref{eq:X-X}, \eqref{D-conc} and applying the union bound for all $j \in [i-1]$, we get that
\[
	 \Pr\Big(\max_{ j \in [i-1]} \left| X_j  - X_j^{(i)} \right|  >  \eps b_n \Big)
	 \leq n  e^{-\omega(\log n)} =e^{-\omega(\log n)}  .
\]
From \eqref{eq:F}, we find that $ \Pr(X_i > a_n + b_n x)= \Omega(n^{-1}) \gg e^{-\omega(\log n)}$ uniformly over all $i\in [n]$. Thus, we get the desired $X_j^{(i)}$ satisfying all conditions of Lemma \ref{L:mixing}. This completes the proof.
\end{proof}

\begin{remark}
The binomial random graph $\mathcal G_{n, p}$ is a special case of $\HG_{n, k, p}$ for $k = 2$.
In the particular case, Theorem~\ref{max_codegree_hyper}
gives the asymptotic distribution  of the maximum degree of $\mathcal G_{n, p}$.  This result  was obtained for the first time by Bollob\'as~\cite{bollobas}  and Ivchenko~\cite{ivchenko} using the method of moments. 
For every $i\in[n]$, they consider the Bernoulli random variable $\eta_i$ that equals 1 if and only if its degree is bigger than $a_n+b_n x$. Letting $\eta=\eta_1+\ldots+\eta_n$, they easily get that $\E\eta\to e^{-x}$ as $n\to\infty$. 
Thus, it is  sufficient to prove that $\eta$ converges in distribution to a Poisson random variable as $n\to\infty$. For $k=2$, one can derive that $\E{\eta\choose r}\to e^{-xr}/r!$ for any fixed  $r\in\mathbb{N}$.   However,
when $k>2$ the dependencies are stronger so the computation of factorial moments becomes much more technically involved. In contrast, our method does not require any computations aside from the single application of the Chernoff bound in (\ref{D-conc}) for all $k$.
%
%
\end{remark}

\begin{remark}
Another advantage of our approach is that it gives an estimate of the rate of convergence to the Gumbel distribution. A careful investigation of the proofs of  Theorem~\ref{Tmain},  Lemma  \ref{L:mixing},
and Theorem \ref{max_degree_hyper} shows that
\begin{align*}
	\biggl|\Pr\left(\max_{i\in [n]} X_i \leq a_n + x b_n\right) - \prod_{i \in [n]}\Pr&\left( X_i\leq a_n + x b_n\right)\biggr|
	\\ &= O\left( \sqrt{\frac{\log^3 n}{\binom{n-1}{k-1} p(1-p)}} +  \sqrt{\frac{k \log^2 n} {n} } \right).
\end{align*}
That is, the rate of convergence is governed by the rate of decrease of $\eps$, for which
$ \Pr\Big(\max_{ j \in [i-1]} \left| X_j  - X_j^{(i)} \right|  >  \eps b_n \Big)$ remains very small.
In addition, for $a_n$ and $b_n$ defined by \eqref{def:ab},  the convergence rate of  $ \prod_{i \in [n]}\Pr\left( X_i\leq a_n + x b_n\right)$ to the Gumbel distribution  is  $O\left(\frac{\log\log n}{\log n}\right)$.
However, this convergence rate can be improved by  using a more precise expression for the scaling parameter $a_n$. 
%
\end{remark}

Our approach  applied to  codegrees $X_S$ in the random hypergraph $\HG_{n, k, p}$ leads  to the following result.

\begin{theorem}\label{max_codegree_hyper}
Assume $p=p(n)\in(0,1)$,
$s =s(n) \in [n-1]$, and
$k=k(n)\in [n]\setminus [s]$ are such that
\begin{equation*}
\binom{n-s}{k-s}p(1-p) \gg s^3 \log^3 n, \qquad (k-s)s^2  \ll (n-s)/\log^2 n.
\end{equation*}
Then $\left[\max_{S\in{[n]\choose s}}X_S-a_n\right]/b_n$ converges in distribution to a standard Gumbel random variable, where
$a_n = a_n\left(\binom{n}{s},  \binom{n-s}{k-s} ,p \right)$ and $b_n = b_n\left(\binom{n}{s},  \binom{n-s}{k-s},
p \right)$ are defined in \eqref{def:ab}.

\end{theorem}
\begin{proof}
	Theorem \ref{max_codegree_hyper} is proved in exactly the same way as
 Theorem \ref{max_degree_hyper}.
 Take any $x\in \Reals$. For all $S\in \binom{[n]}{s}$, let $A_{S}:= \{X_{S} > a_n + xb_n \}$.
Let $d:=\binom{n}{s}$ and  $N :=\binom{n-s}{k-s}$.  Since  $\binom{n}{s} \leq n^s$,
 the assumptions imply $Np(1-p) \gg \log^3 d$.
Recalling  that  $X_{S} \sim \Bin\left(N, p \right) $ and using Lemma~\ref{L:bin},  we find that
\begin{equation*}
\prod_{S \in \binom{[n]}{s}} \Pr( \overline{A}_{S})\rightarrow  e^{-e^{-x}}.
\end{equation*}
Again, we can take $D_S=\varnothing$ for all $S\in{[n]\choose s}$. Thus, we only need to show that $\varphi(\A,\D)=o(1)$. The key fact needed to apply  Lemma~\ref{L:mixing} is that, for any fixed $\eps>0$,
\begin{equation*}
 	\Pr \left( \big|X_{U\cup S} - \E{X_{U\cup S}}\big| \leq \eps b_n   \text{ for all distinct $U,S \in \binom{[n]}{s}$} \right)
 	\geq 1- e^{-\omega(\log d)}.
 \end{equation*}
 Similarly to \eqref{D-conc}, this is a straightforward application of the Chernoff bound.
 \end{proof}

\subsection{Maximum clique-extension counts}
\label{max_clique_result}

Let $k\geq 3$ be an integer. In this section, we find the \purple{asymptotic distribution} 
of the maximum number of $k$-clique extensions in the random graph $\mathcal G_{n, p}$. 
For $i\in[n]$, let $X_i$ be the number of $k$-cliques containing \red{vertex} $i$. 
Below, we show that Theorem~\ref{Tmain} implies the \purple{asymptotic} distribution of the maximum value of $X_i$ over $i\in[n]$.


Let
\begin{align*}
a^k_n  & :=\frac{(pn)^{k-2}p^{{k-1\choose 2}}}{(k-1)!}\left[pn+(k-1)\sqrt{2np(1-p)\log n}\left(1 - \frac{\log\log n}{4\log n} - \frac{\log( 2 \sqrt \pi)}{2\log n}\right)\right],\\
b^k_n & := \frac{1}{(k-2)!}(pn)^{k-2}p^{k-1\choose 2}\sqrt{\frac{ np(1-p) }{2 \log n} }.
\end{align*}

\begin{theorem}
Let $p=p(n)\in(0,1)$, $k=k(n)\in\{3,\ldots,n\}$ be such that
\begin{equation}
\label{log_bound}
\log^3 n=o(np(1-p)),\quad \log^2 n=o\left(\frac{np^{{k-1\choose 2}+1}(1-p)}{k^3}\right).
\end{equation}
Then $\left[\max_{i\in[n]}X_i-a_n^k\right]/b_n^k$ converges in distribution to a standard Gumbel random variable.  
\label{ext thm}
\end{theorem}
\begin{proof}
Let $d_i$ be the degree of the vertex $i$, and $Y_i = \E{ X_i ~\middle|~ d_i }= \binom{d_i}{k-1} p^{\binom{k-1}{2}}$. Note that
\[
\max_{i \in [n]} Y_i = \binom{\max_{i \in [n]} d_i}{k-1} p^{\binom{k-1}{2}}.
\]
\red{Let $x\in\mathbb{R}$. 
By Theorem~\ref{max_degree_hyper}, we have
\begin{align*}
\P{ \max_{ i \in [n]}
Y_i \le \binom{a_n + b_n x} {k - 1} p^{ \binom{k-1}{2} }}
= \P{ \max_{i \in [n]} d_{i} \le a_{n} + b_{n}x }
\to e^{-e^{-x}},
\end{align*}
where $a_n = a_n(n, n - 1, p)$ and $b_n = b_n(n, n - 1, p)$ are defined in \eqref{def:ab}. Computing directly, we get
$$
\binom{a_n+b_n x}{k-1} p^{ \binom{k-1}{2}}=a^k_n+b^k_n x (1+o(1)),
$$ 
implying that 
\begin{equation}
\P{ \max_{i \in [n]} Y_i \leq  a^k_n + b^k_n x } \to e^{-e^{-x}}.
\label{cond_expect_Gumbel}
\end{equation}
}
Set $\widetilde X_i=\frac{X_i-a^k_n}{b^k_n}$, $\widetilde Y_i=\frac{Y_i-a^k_n}{b^k_n}$. It remains to show that
\begin{align}
\P{ \left| \widetilde X_i - \widetilde Y_i \right| > \eps } = \o{ \P{ Y_i > a_n^k + x b_n^k } }=\o{\frac{1}{n}},
\label{pro}
\end{align}
and apply Lemma \ref{L:transfer}.

The de Moivre--Laplace theorem and the relation $\int_x^{\infty} e^{-t^2/2}dt=\frac{1}{x}e^{-x^2/2}(1+o(1))$ (see, e.g.,~\cite[Relation (1')]{Bol_sequences}) imply
\red{
$$
 \Pr\left(|d_i-np|>\sqrt{2np(1-p)\log n}\right)
= \frac{1+o(1)}{n\sqrt{\pi\log n}}.
$$
}
Therefore,
\begin{align*}
\P{ \left| \widetilde X_i - \widetilde Y_i \right| > \eps } 
&= \P{|X_i-Y_i|>\eps b_n^k}\\
&=~\sum_{|j-np|\leq\sqrt{2np(1-p)\log n}}\P{\left|X_i-\binom{j}{k-1} p^{\binom{k-1}{2}}\right|>\eps b_n^k,\, d_i=j}+o\left(\frac{1}{n}\right).
\end{align*}

It remains to bound from above $\P{\left.X_i-\binom{j}{k-1} p^{\binom{k-1}{2}}>\eps b_n^k~\right|~d_i=j}$ 
and $\P{\left.X_i-\binom{j}{k-1} p^{\binom{k-1}{2}}<-\eps b_n^k~\right|~d_i=j}$. For the lower tail, we apply Janson's inequality~\cite[Theorem 2.14]{JLR} that does not work, in general, for upper tails. However, a weaker bound~\cite[Proposition 2.44]{JLR} can be applied for that. To apply the bounds, we need to compute the number of $(k-1)$-cliques that are not edge-disjoint with a given $(k-1)$-clique in $K_j$ (which is denoted by $\Delta$ below) and the expected number of pairs of non-edge-disjoint $(k-1)$-cliques in $\mathcal{G}_{j,p}$ (which is denoted by $\overline{\Delta}$ below).\\

For $j\in\mathbb{N}$ such that $|j-np|\leq\sqrt{2np(1-p)\log n}$, denote the number of $(k-1)$-subsets of $[j]$ having at least 2 common element with $[k-1]$ by $\Delta$. Clearly,
$$
\Delta={j\choose k-1}-{j-k+1\choose k-1}-(k-1){j-k+1\choose k-2}={k-1\choose 2}\frac{j^{k-3}}{(k-3)!}(1+o(1)).
$$
Moreover, let $\overline{\Delta}$ be the expected number of pairs of 
\purple{(not necessarily distinct)} $k$-cliques with non-empty edge intersections:
$$
\begin{aligned}
\overline{\Delta}={j\choose k-1}\sum_{\ell=2}^{k-1}{k-1\choose\ell}{j-k+1\choose k-1-\ell}p^{(k-1)(k-2)-{\ell\choose 2}}=\\
\frac{j^{2k-4}}{(k-1)!(k-3)!}{k-1\choose 2}p^{(k-1)(k-2)-1}(1+o(1)).
\end{aligned}
$$
By \eqref{log_bound} and \cite[Proposition 2.44]{JLR}, 
uniformly over all $j\in\mathbb{N}$ such that $|j-np|\leq\sqrt{2np(1-p)\log n}$, we have
\begin{multline}
 \P{\left.X_i-\binom{j}{k-1} p^{\binom{k-1}{2}}>\eps b_n^k\,\right| d_i=j}\leq\\
 (\Delta+1)\exp\left[-\frac{\varepsilon^2 \left[b_n^k\right]^2}{4(\Delta+1)(\E{X_i|d_i=j}+\eps b_n^k/3)}\right]=\\
 \exp\left[-\frac{\varepsilon^2 p^{{k-1\choose 2}}np(1-p)}{4(k-2)^2\log n}(1+o(1))\right]=o\left(\frac{1}{n}\right).
\label{expect_Janson_above}
\end{multline}
Moreover, by \eqref{log_bound} and Janson's inequality~\cite[Theorem 2.14]{JLR}, uniformly over all $j\in\mathbb{N}$ such that $|j-np|\leq\sqrt{2np(1-p)\log n}$,
\begin{equation}
\begin{aligned}
 \P{\left.X_i-\binom{j}{k-1} p^{\binom{k-1}{2}}<-\eps b_n^k\,\right|~d_i=j}
 \leq\exp\left[-\frac{\varepsilon^2 \left[b_n^k\right]^2}{2\overline{\Delta}}\right]=\\
 \exp\left[-\frac{\eps^2 np^2(1-p)}{2(k-2)^2\log n}(1 \red{-} o(1))\right]=o\left(\frac{1}{n}\right).
\end{aligned}
\label{expect_Janson_below}
\end{equation}

Finally, combining (\ref{expect_Janson_above}) and (\ref{expect_Janson_below}), we get
$$
\sum\limits_{|j-np|\leq\sqrt{2np(1-p)\log n}} \P{\left|X_i-\binom{j}{k-1} p^{\binom{k-1}{2}}\right|>\eps b_n^k,\, d_i=j}=o\left(\frac{1}{n}\right).
$$
\end{proof}



\subsection{Maximum number of $h$-neighbours}\label{S:max-h-neighbours}


The particular case of the \red{following} result for constant $h$ was proved in \cite{rodionov2018distribution}. Let us show that it is a more or less direct corollary of Theorem~\ref{Tmain}.

For $h\in\mathbb{N}$ and $\mathbf{x} \in {[n]\choose h}$, denote the number of common neighbours of vertices in $\mathbf{x}$ in $\mathcal{G}_{n,p}$ by $X_{\mathbf{x}}$. Set $a_{h,n}:=a_n({n\choose h},n,p^h)$, $b_{h,n}:=b_n({n\choose h},n,p^h)$, where $a_n$ and $b_n$ are defined in~(\ref{def:ab}).

\begin{theorem}
Let $h=h(n)=o(\log n/\log\log n)$ and $p = p(n)\in(0,1)$ be such that
\begin{equation}
\frac{p^h}{h^3}\gg\frac{\log^3 n}{n},\quad
1-p\gg\sqrt{\frac{\log\log n}{\log n}}.
\label{common_neighb_cond}
\end{equation}
Then $\left[\max_{\mathbf{x}\in{[n]\choose h}}X_{\mathbf{x}}-a_{h,n}\right]/b_{h,n}$ converges in distribution to a standard Gumbel random variable.  
\label{common_neighb_thm}
\end{theorem}

\begin{proof}
For any $\mathbf{x}\in \binom{[n]}{h}$, $X_{\mathbf{x}}$ follows $\mathrm{Bin}\left(n-h, p^h \right)$. Then, by Lemma~\ref{L:bin},
\begin{equation}
\prod_{\mathbf{x}\in \binom{[n]}{h}} \P{ X_{\mathbf{x}} \le a_{h,n} + b_{h,n} x } \rightarrow e^{-e^{-x}}.
\label{extensions_Gumbel}
\end{equation}

Let us label 
the set ${[n]\choose h}$ by positive integers $1,2,\ldots,{n\choose h}$. Set $d={n\choose h}$ and fix $x\in\mathbb{R}$. For $i\in[d]$, set $A_i=\{X_i>a_{h,n}+b_{h,n}x\}$, $D_i=[d]\setminus D_i^*$, where $D_i^*$ 
\purple{is the set of $h$-subsets of $[n]$ that do not intersect with the $i$-th set 
(the $i$-th set is simply denoted by $i$ in what follows)}.

Let us first verify that $\varphi(\A,\D)=o(1)$. Let $i\in[d]$. We \red{denote by} $\obj_i$ the set of edges of $\mathcal{G}_{n,p}$ that \red{do} not contain vertices of $i$. For any $j \in [i-1]\setminus D_i$, let $X_{j,i}$ be the number of vertices in $i$ adjacent to all vertices in $j$ in $\mathcal{G}_{n,p}$ (notice that $i$ and $j$ are disjoint). Since $X_{j,i}$ is independent of $\obj_{i}$, we get $X_j - \E{ X_j ~\middle|~ \obj_i}= X_{j, i} - \E{ X_{j, i} }.$ Set $\widetilde X_i=\frac{X_i-a_{h,n}}{b_{h,n}}$, $\widetilde X^{(i)}_{j}= \E{ \widetilde X_j ~\middle|~ \obj_i }$. Since $X_{j,i} \sim \mathrm{Bin}\left( h, p^h \right),$
we get by (\ref{common_neighb_cond}) and the Chernoff bound (see, e.g.,~\cite[Theorem 2.1]{JLR}) that, for every $\varepsilon>0$,
\begin{equation}
\begin{aligned}
\P{\left|\widetilde X_{j}  -  \widetilde X^{(i)}_{j} \right| > \varepsilon}=\P{\left|X_{j,i}-hp^h\right|>\varepsilon b_{h,n}}\leq 2\exp\left[-\frac{(\varepsilon b_{h,n})^2}{2(hp^h+\varepsilon b_{h,n}/3)}\right]=\\
\exp\left[-\frac{3\varepsilon b_{h,n}}{2}(1+o(1))\right]\leq\exp\left[-\frac{3\varepsilon}{2}\sqrt{\frac{np^h(1-p)}{2h\log n}}(1+o(1))\right]=o\left(\frac{1}{n^{2h}}\right).
\end{aligned}
\end{equation}
By (\ref{extensions_Gumbel}), we get $\P{\widetilde X_i > x} = {n\choose h}^{-1} e^{-x}(1+o(1))$. Therefore, by the union bound,
\begin{align*}
\Pr\left(\max_{j \in [i-1]\setminus D_i} \left|\widetilde X_{j}  -  \widetilde X_{j}^{(i)} \right| > \varepsilon\right)  \le & {n\choose h} \Pr\left(\left|\widetilde X_{j}  -  \widetilde X_{j}^{(i)} \right| > \varepsilon\right)=\\ & o\left(\frac{1}{{n\choose h}}\right)= o(1)\P{\widetilde X_{i} > x}.
\end{align*}
Lemma \ref{L:mixing} implies $\varphi(\A,\D)=o(1)$.\\

By Corollary~\ref{Cmain}, it remains to verify the conditions $\Delta_1(\A,\D)=o(1)$ and $\Delta_2(\A,\D)=o(1)$. Unfortunately, these conditions do not hold. Nevertheless, the events $( A_i )_{i \in [d]}$ can be modified slightly to make the desired relations hold. Define
$$
 E=\bigcap_{\ell=1}^{h-1}\bigcap_{\mathbf{u}\in{[n]\choose\ell}}\left\{ X_{\mathbf{u}} \le n p^{\ell} +  \sqrt{  2 \ell n p^{\ell} (1-p^{\ell}) \log n } \right\}.
$$
For $i\in[d]$, let $\widetilde A_i=A_i\cap E$
and $\widetilde\A=(\widetilde A_i)_{i\in[d]}$.\\

The following lemma is proven in~\cite{rodionov2018distribution} for constant $h$; for $h=o(\log n/\log\log n)$ the same proof works.

\begin{lemma}[\cite{rodionov2018distribution}]
The following relations hold
\begin{enumerate}
\item $\P{E}=1-o(1)$,
\item for every $x\in\mathbb{R}$, $\P{\widetilde A_i}=(1-o(1))\P{A_i}$ uniformly over all $i\in[d]$,
\item $\sum_{i\in[d]}\sum_{j\in[i-1]\cap D_i}\P{\widetilde A_i\cap\widetilde A_j} = \o{ 1 }$.
\end{enumerate}
\label{lm comm neib}
\end{lemma}

By Lemma~\ref{lm comm neib} and since $\varphi(\A,\D)=o(1)$, uniformly over all $i\in[d]$,
\begin{align*}
&\left|\Pr\left(\bigcup_{j\in[i-1]\setminus D_i}\widetilde A_j~\middle|~\widetilde A_i\right)-\Pr\left(\bigcup_{j\in[i-1]\setminus D_i}\widetilde A_j\right)\right| \\
&\leq~\left|
 \Pr\left(\bigcup_{j\in[i-1]\setminus D_i}\widetilde A_j~\middle|~ A_i\right)\frac{\Pr(A_i)}{\Pr(\widetilde A_i)}-\Pr\left(\bigcup_{j\in[i-1]\setminus D_i}A_j\right)\right|+\Pr(\overline{E}) \\
&\leq~\left|
 \Pr\left(\bigcup_{j\in[i-1]\setminus D_i}A_j~\middle|~ A_i\right)\frac{\Pr(A_i)}{\Pr(\widetilde A_i)}-\Pr\left(\bigcup_{j\in[i-1]\setminus D_i}A_j\right)\right|+\frac{\Pr(A_i\cap\overline{E})}{\Pr(\widetilde A_i)}+\Pr(\overline{E})\\ 
&\leq~\left|
 \Pr\left(\bigcup_{j\in[i-1]\setminus D_i}A_j~\middle|~ A_i\right)(1+o(1))-\Pr\left(\bigcup_{j\in[i-1]\setminus D_i}A_j\right)\right|+\frac{\Pr(A_i)}{\Pr(\widetilde A_i)}-1+\Pr(\overline{E})=o(1).
\end{align*}
Therefore, $\varphi(\widetilde\A,\D)=o(1)$ as well.\\

Notice that the third statement of Lemma~\ref{lm comm neib} is exactly $\Delta_1'(\widetilde\A,\D)=o(1)$ (see the definitions of $\Delta_1'$ and $\Delta_2'$ in Section~\ref{Sec:related}). It remains to prove that $\Delta_2'(\widetilde\A,\D)=o(1)$. But this is straightforward:
$$
\Delta_2'(\widetilde\A,\D)=\sum_{i\in[d]}\sum_{j\in[i-1]\cap D_i}\Pr(\widetilde A_i)\Pr(\widetilde A_j)\leq \sum_{i\in[d]}\sum_{j\in[i-1]\cap D_i}\Pr(A_i)\Pr(A_j)\leq
$$
$$
{n\choose h}^{-1}e^{-2x}(1+o(1))\max_{i\in[d]}|D_i|=\frac{{n\choose h}-{n-h\choose h}}{{n\choose h}}e^{-2x}(1+o(1))=o(1).
$$
By Theorem~\ref{Tmain}, we get that (\ref{a_e_indep_formula_for_A}) holds for $\widetilde\A$. The first two statements of Lemma~\ref{lm comm neib} imply 
\purple{that (\ref{a_e_indep_formula_for_A}) also holds} for $\A$. Indeed, $\Pr(E)=1-o(1)$ implies that
$$
\Pr\left(\cap_{i\in[d]}\overline{A_i}\right)=\Pr\left(\cap_{i\in[d]}\overline{\widetilde A_i}\right)-\Pr\left(\overline{E}\setminus\left(\cap_{i\in [d]}\overline{A_i}\right)\right)=\Pr\left(\cap_{i\in[d]}\overline{\widetilde A_i}\right)+o(1);
$$
and $\Pr(\widetilde A_i)=(1-o(1))\Pr(A_i)$ implies
$$
\prod_{i\in[d]}\Pr\left(\overline{\widetilde A_i}\right)=\prod_{i\in[d]}\left[1-\Pr(A_i)(1+o(1))\right]=\left(1-\frac{e^{-x}+o(1)}{d}\right)^d\to e^{-e^{-x}}.
$$

\end{proof}

\subsection{Further results in maximum extensions counts}
\label{extensions_further}

As we discussed in the beginning of Section~\ref{rg_applications}, the above results are in the framework of extensions counting. Given a strictly balanced grounded pair $(H,G)$ with $|V(H)|=h$, we are interested in the \purple{asymptotic} behaviour of $\max_{\mathbf{x}\in{[n]_h}}X_{\mathbf{x}}$. Recall that, in~\cite{spencer_extensions}, Spencer proved the law of large numbers~(\ref{extensions_lln}). In recent paper~\cite{vsileikis2022counting}, \v{S}ileikis and Warnke studied the validity of this law when $\mu=\Theta(\log n)$.


In Section~\ref{max_clique_result}, we found an optimal denominator in the the law of large numbers for $h=1$, $G=K_k$ and $p$ satisfying~(\ref{log_bound}) (i.e. far from the threshold value):
$$
\frac{\max_{i\in[n]} X_i-\mu }{\mu(k-1)\sqrt{2(1-p)\log n/(pn)}}\stackrel{\Pr}\rightarrow 1.
$$
Notice that the result holds for the numerator $\max_{i\in[n]}|X_i-\mu|=\max\{\max_{i\in[n]}X_i-\mu,\mu-\min_{i\in[n]}X_i\}$ as well. Indeed, let $d_i$ be the degree of the vertex $i$. Theorem~\ref{max_degree_hyper} implies the asymptotic distribution of the minimum degree of $\mathcal{G}_{n,p}$ since it equals in distribution to $n-\max_{i\in[n]}d_i[G_{n,1-p}]$. Thus, 
$$
\Pr\left(\min_{i\in[n]}\E{X_i|d_i}\geq \widetilde a_n^k-b_n^k x\right)\to e^{-e^{-x}}
$$
where $\widetilde a_n^k=\frac{1}{(k-1)!}(pn)^{k-2}p^{{k-1\choose 2}}n-a_n^k$. To get the distribution of the minimum degree, it remains to reformulate Lemma~\ref{L:transfer} for the events $A_i:=\{X_i<\widetilde a_n^k-b_n^k x\}$ and probabilities $\Pr(\min X_i\geq\widetilde a_n^k-b_n^k x)$, $\Pr(\min Y_i\geq\widetilde a_n^k-b_n^k x)$  (clearly, the same proof works) and follow absolutely the same steps as in the proof of Theorem~\ref{ext thm}.

Our method works not only in the case $h=1$. In Section~\ref{S:max-h-neighbours}, we have found the \purple{asymptotic} distribution of $X_{\mathbf{x}}$ when $h\geq 2$ and $G$ contains a unique vertex outside $H$ which is adjacent to all vertices in $H$.  Our arguments should work even in the case when $H,G$ are both cliques of arbitrary size. Indeed, the result for cliques $G$ such that $|V(G)|-|V(H)|\geq 2$ can be obtained from Theorem~\ref{common_neighb_thm} using Lemma~\ref{L:transfer} in the same way as we obtain Theorem~\ref{ext thm} from Theorem~\ref{max_degree_hyper} in Section~\ref{max_clique_result}.




\section{Acknowledgements}


M. Isaev is supported by the Australian Research Council Discovery Project DP190100977 and Australian Research Council Discovery Early Career Researcher Award DE200101045. M. Zhukovskii is supported by the Israel Science Foundation grant 2110/22.

\appendix

\section{Bridging sequences: proofs.}
\label{bridges}

%

\subsection{Proof of Lemma \ref{L:mixing}}\label{S:mixing}

Find $\delta>0$ such that $0<F(x-\delta)\leq F(x+\delta)<1$.
Let $\varepsilon\in(0,\delta/2)$.  We may assume that $n$ is so large that $\Pr(A_i)\leq\Pr(A_i(x-2\varepsilon))<1$ for all $i\in[d]$ (otherwise, $\prod_{i=1}^d\Pr(\overline{A_i(x-2\varepsilon)})$ can not approach $F(x-2\varepsilon)$).
For  $i\in [d]$ and $j\in [i-1]\setminus D_i$,  consider  the events $A^{\eps}_{i}:=A_i(x+2\varepsilon)$ and $U_{ji}^{\eps}:= \{ X_j^{(i)} > a_n + (x+\eps)b_n\}$. Then, from (\ref{lem:eq:independent}), we get that uniformly over all  $i\in [d]$
\[
	\Pr\left( \bigcup_{j \in [i-1]\setminus D_i} \left(U^\eps_{ji} \setminus A_j  \right)\right)  = o(1)\Pr(A_i)
	 \quad  \text{and} \quad
\Pr\left( \bigcup_{j \in [i-1]\setminus D_i} \left(A_j^{\eps} \setminus U^\eps_{ji} \right)\right) = o(1)\Pr(A_i).
\]
The events  $U^{\eps}_{ji} $ and $A_i$ are independent since $X_j^{(i)}$ is independent of $X_i$. Therefore,
%
\begin{align*}
		\Pr \left( \bigcup_{j \in [i-1]\setminus D_i} A_j  ~\middle|~ A_i\right)
		&~\geq \Pr \left( \bigcup_{j \in [i-1]\setminus D_i} U^{\eps}_{ji} ~\middle|~ A_i \right)  -o(1)
		\\&=~\Pr \left( \bigcup_{j \in [i-1]\setminus D_i} U^{\eps}_{ji}\right) - o(1)
				\geq\Pr \left( \bigcup_{j \in [i-1]\setminus D_i} A^{\eps}_{j}\right) - o(1).
\end{align*}
By the union bound, we get  that
\begin{align*}
	\Pr \left( \bigcup_{j \in [i-1]\setminus D_i} A_j \right)  - \Pr \left( \bigcup_{j \in [i-1]\setminus D_i} A_j ~\middle|~ A_i \right)
	&\leq
	\Pr \left( \bigcup_{j \in [i-1]\setminus D_i} A_j\right) - \Pr \left( \bigcup_{j \in [i-1]\setminus D_i} A^{\eps}_j\right) + o(1)
	\\
	\leq\sum_{j\in [i-1]\setminus D_i}\Pr\left(A_j\setminus\bigcup_{s\in [i-1]\setminus D_i} A^{\eps}_s\right) + o(1)
	&\leq  \sum_{j\in [i-1]\setminus D_i} \Pr (A_j \setminus A_j^\eps) + o(1).
\end{align*}
Using the inequality $\sum_{i\in [d]}t_i \leq -1+ \prod_{i\in[d]}(1+t_i)$, where $t_i:=
\frac{\Pr(A_i) - \Pr(A_i^{\eps})}{1-\Pr(A_i)}\geq 0,$
and recalling that  $F(x)>0$,  we estimate
\[
	\sum_{i\in [d]} \Pr (A_i \setminus A_i^\eps)
	 \leq \sum_{i\in [d]} \frac{\Pr (A_i) -  \Pr(A_i^\eps)}{1- \P{A_i} }
	\leq   -1 + \prod_{i \in [d]} \frac{ 1- \Pr(A_i^{\eps})} {1-\Pr(A_i)}  \rightarrow
	 \frac{F(x+2\eps)}{F(x)}-1.
\]
Recalling that $F$ is continuous at $x$ and that
the above holds for any $\eps\in(0,\delta/2)$, we conclude that
\[
	\Pr \left( \bigcup_{j \in [i-1]\setminus D_i} A_j \right)  - \Pr \left( \bigcup_{j \in [i-1]\setminus D_i} A_j ~\middle|~ A_i \right) \leq o(1).
\]
The lower bound
\[
	\Pr \left( \bigcup_{j \in [i-1]\setminus D_i} \red{A_j} \right)  - \Pr \left( \bigcup_{j \in [i-1]\setminus D_i} A_j ~\middle|~ A_i \right) \geq o(1)
\]
is obtained similarly by using the events $A_j^{-\eps}:= A_i(x-2\varepsilon)$, $U_{ji}^{-\eps}:=\{X_j^{(i)}>a_n + (x-\varepsilon) b_n\}$ and the relations
\[
	\Pr\left( \bigcup_{j \in [i-1]\setminus D_i} \left(A_j\setminus U^{-\eps}_{ji}  \right)\right)  = o(1)\Pr(A_i),\quad
\Pr\left( \bigcup_{j \in [i-1]\setminus D_i} \left(U^{-\eps}_{ji}\setminus A_j^{-\eps} \right)\right) = o(1)\Pr(A_i),
\]
that hold uniformly over all  $i\in [d]$. This completes the proof of Lemma \ref{L:mixing}.


\subsection{Proof of Lemma \ref{L:transfer}}\label{S:transfer}




Find $\delta>0$ such that $0<F(x-\delta)\leq F(x+\delta)<1$. Let
$\eps\in(0,\delta)$.
Let
  $A_i^{\eps}:= A_i(x+\eps)$, $B_i:= \{ Y_i >a_n+b_nx\}$. 
\red{From assumption (i), we get
\begin{align*}
   1 - \Pr\Big(\bigcup_{i\in [d]} A_i^{\eps}\Big) \rightarrow F(x+\eps).
\end{align*}
Also by the third assumption from the list of preliminary assumptions at the beginning of Subsection \ref{S:bridging}, we obtain
\begin{align*}
   \prod_{i \in [d]} \left(1- \Pr(A_i^{\eps}) \right)  \rightarrow F(x+\eps).
\end{align*}
Therefore, we conclude
\begin{equation}\label{eq:i}
   1 - \Pr\Big(\bigcup_{i\in [d]} A_i^{\eps}\Big)  \sim
   \prod_{i \in [d]} \left(1- \Pr(A_i^{\eps}) \right)  \rightarrow F(x+\eps).
\end{equation}
}
Since $F(x+\eps)\geq F(x-\eps)>0$, we get
\begin{equation}
\sum_{i \in [d]} \Pr(A_i^{\eps}) \leq -\sum_{i \in [d]} \log \left(1- \Pr(A_i^{\eps}) \right)  =O(1).
\label{eq:big-O}
\end{equation}
From (ii), we find that $\Pr(A_i^{\eps}\setminus B_i) =  o(1) \Pr(A_i^{\eps}).$
Then the relations
\begin{eqnarray*}
 \Pr\Big(\bigcup_{i\in [d]}A_i^{\eps}\Big)-\Pr\Big(\bigcup_{i\in [d]}B_i\Big)&\leq&
 \Pr\Big(\bigcup_{i\in [d]}A_i^{\eps}\setminus\bigcup_{i\in [d]}B_i\Big)
 \\ & \leq& \sum_{i\in[d]}\Pr\Big(A_i^{\eps}\setminus\bigcup_{j\in [d]}B_j\Big)\leq
 \sum_{i\in[d]}\Pr\Big(A_i^{\eps}\setminus B_i\Big)
\end{eqnarray*}
imply
\[
	\Pr\Big(\bigcup_{i\in [d]} B_i\Big) \geq
	\Pr\Big(\bigcup_{i\in [d]} A_i^{\eps}\Big) - o(1) \sum_{i\in [d]} \Pr(A_i^\eps)   = 1- F(x+\eps) -o(1).
\]
The last equality follows from \eqref{eq:i} and (\ref{eq:big-O}). 
Recalling that $F$ is continuous and that
the above  holds for any $\eps\in(0,\delta)$, we conclude that $1 - \Pr\Big(\bigcup_{i\in [d]} B_i\Big)  \leq F(x) +o(1)$.
The lower bound $1 - \Pr\Big(\bigcup_{i\in [d]} B_i\Big)  \geq F(x)-o(1)$
is obtained similarly, using the events  $A_i^{-\eps}= A_i(x-\eps)$ and the relations $\Pr(B_i\setminus A_i^{-\eps}) =  o(1) \Pr(A_i^{-\eps})$ that follow directly from (ii).

\end{document}